\documentclass[a4paper,11pt]{article}

\usepackage{amsmath,amsthm,amssymb,mathrsfs,latexsym,amsfonts}
\usepackage{graphicx,psfrag,epsfig}
\usepackage{bbm}
\usepackage[english]{babel}
\usepackage[latin1]{inputenc}
\usepackage{a4wide}

\newtheorem{theorem}{Theorem}[section]
\newtheorem{lemma}[theorem]{Lemma}
\newtheorem{proposition}[theorem]{Proposition}

\newtheorem{corollary}[theorem]{Corollary}

\newtheorem{remark}[theorem]{\it Remark}

\newcounter{paraga}[section]
\renewcommand{\theparaga}{{\bf\arabic{paraga}.}}
\newcommand{\paraga}{\medskip \addtocounter{paraga}{1} 
\noindent{\theparaga\ } }

\begin{document}

\def\MP{\,{<\hspace{-.5em}\cdot}\,}
\def\SP{\,{>\hspace{-.3em}\cdot}\,}
\def\PM{\,{\cdot\hspace{-.3em}<}\,}
\def\PS{\,{\cdot\hspace{-.3em}>}\,}
\def\EP{\,{=\hspace{-.2em}\cdot}\,}
\def\PP{\,{+\hspace{-.1em}\cdot}\,}
\def\PE{\,{\cdot\hspace{-.2em}=}\,}
\def\N{\mathbb N}
\def\C{\mathbb C}
\def\Q{\mathbb Q}
\def\R{\mathbb R}
\def\T{\mathbb T}
\def\A{\mathbb A}
\def\Z{\mathbb Z}
\def\demi{\frac{1}{2}}

\begin{titlepage}
\author{Abed Bounemoura~\footnote{abedbou@gmail.com, Centre de Recerca Matemàtica, Campus de Bellaterra, Edifici C, 08193, Bellaterra}}
\title{\LARGE{\textbf{Normal forms, stability and splitting of invariant manifolds I. Gevrey Hamiltonians}}}
\end{titlepage}

\maketitle

\begin{abstract}
In this paper, we give a new construction of resonant normal forms with a small remainder for near-integrable Hamiltonians at a quasi-periodic frequency. The construction is based on the special case of a periodic frequency, a Diophantine result concerning the approximation of a vector by independent periodic vectors and a technique of composition of periodic averaging. It enables us to deal with non-analytic Hamiltonians, and in this first part we will focus on Gevrey Hamiltonians and derive normal forms with an exponentially small remainder. This extends a result which was known for analytic Hamiltonians, and only in the periodic case for Gevrey Hamiltonians. As applications, we obtain an exponentially large upper bound on the stability time for the evolution of the action variables and an exponentially small upper bound on the splitting of invariant manifolds for hyperbolic tori, generalizing corresponding results for analytic Hamiltonians.
\end{abstract}
 
\section{Introduction and main results}\label{s1}

\paraga Let $n\geq 2$ be an integer, $\T^n=\R^n/\Z^n$ and $B_R$ be the closed ball in $\R^n$, centered at the origin, of radius $R>0$ with respect to the supremum norm. For $\varepsilon\geq 0$, we consider a near-integrable Hamiltonian of the form
\begin{equation*}
\begin{cases} 
H(\theta,I)=h(I)+f(\theta,I) \\
|f| \leq \varepsilon <\!\!<1
\end{cases}
\end{equation*}
where $(\theta,I) \in \mathcal{D}_R=\T^n \times B_R$ are angle-action coordinates for the integrable part $h$ and $f$ is a small perturbation in some suitable topology defined by a norm $|\,.\,|$. The phase space $\mathcal{D}_R$ is equipped with the symplectic structure induced by the canonical symplectic structure on $\T^n \times \R^n=T^*\T^n$. For $\varepsilon=0$, the system is trivially integrable: the action variables are integrals of motion and their level sets $\{I=I_0\}$, $I_0 \in B_R$, are invariant embedded tori on which the flow is linear with frequency $\nabla h(I_0)$. An important subject of study is the dynamical properties of these systems when $\varepsilon>0$. 

Here we will be interested in the local dynamics of the perturbed system around an invariant torus of the integrable system, which, without loss of generality, we assume is located at $I=0$. The qualitative and quantitative properties of this invariant torus with a linear flow is then determined by its frequency vector $\omega=\nabla h(0) \in \R^n$. Let us say that a vector subspace of $\R^n$ is rational if it has a basis of vectors with rational (or equivalently, integer) components, and we let $F=F_\omega$ be the smallest rational subspace of $\R^n$ containing $\omega$. If $F=\R^n$, the vector $\omega$ is said to be non-resonant and the dynamics on the invariant torus $\mathcal{T}_0=\{I=0\}$ is then minimal and uniquely ergodic. If $F$ is a proper subspace of $\R^n$ of dimension $d$, the vector $\omega$ is said to be resonant and $d$ (respectively $m=n-d$) is the number of effective frequencies (respectively the multiplicity of the resonance): the invariant torus $\mathcal{T}_0$ is then foliated into invariant $d$-dimensional tori (which are just images of translates of $F$ under the canonical projection $\pi_0 : \R^n \rightarrow \mathcal{T}_0$) on which the dynamics is again minimal and uniquely ergodic. The example of a resonant vector to keep in mind, and to which the general case can be actually reduced, is $\omega=(\varpi,0)\in \R^n=\R^d \times \R^m$ with $\varpi \in \R^d$ non-resonant, for which $F=\R^d \times \{0\}$. From now on, we will assume that $1\leq d \leq n$ (the case $d=0$ corresponds to an invariant torus which consists uniquely of equilibrium solutions), and the particular case $d=1$ will play a special role: in this case, writing $\omega=v$, the torus is foliated by periodic orbits and if $T$ denotes the minimal common period, then $Tv \in \Z^n$, and such a vector $v$ will be called periodic.

In order to study the local dynamics of the perturbed system around such an invariant torus, it is important to quantify the minimal character of this linear flow with frequency $\omega$ (on the whole invariant torus if $d=n$ or on each leaf if $d<n$). In general, to such a vector $\omega$ one can associate a constant $Q_\omega>0$ and a real-valued function $\Psi_\omega$ defined for all real numbers $Q\geq Q_\omega$, which is non-decreasing and unbounded,  by
\begin{equation}\label{psi}
\Psi_\omega(Q)=\max\left\{|k\cdot\omega|^{-1} \; | \; k\in \Z^n \cap F , \; 0<|k|\leq Q \right\}
\end{equation}
where $\cdot$ denotes the Euclidean scalar product and $|\,.\,|$ is the supremum norm for vectors (see for instance \cite{BF12} for much more detailed information, where this function is denoted by $\Psi'_\omega$). Note that by definition, we have
\[ |k\cdot\omega|\geq \frac{1}{\Psi_\omega(Q)}, \quad k\in \Z^n \cap F, \; 0<|k|\leq Q. \]
Let us say that a vector $\omega$ is Diophantine if there exist constants $\gamma>0$ and $\tau \geq d-1$ such that $\Psi_\omega(Q)\leq \gamma^{-1}Q^\tau$, and let us denote by $\Omega_d(\gamma,\tau)$ the set of such vectors. For $d=1$, $\omega=v$ is $T$-periodic and $F=F_v$ is just the real line generated by $v$, so that for any non-zero vector $k \in \Z^n \cap F$, as $k$ is collinear to $v$ and $k$ and $Tv$ are non-zero integer vectors, we have
\[  |k.v|^{-1}=|k|^{-1}|v|^{-1}=T|k|^{-1}|Tv|^{-1}\leq T \]
hence $\Psi_v(Q)\leq T$ and therefore any $T$-periodic vector belongs to $\Omega_1(T^{-1},0)$.

\paraga For an analytic Hamiltonian system, it is well-known that in the neighbourhood of an unperturbed invariant torus with a Diophantine frequency vector, the system can be analytically conjugated to a simpler system where the perturbation has been split into two parts: a resonant part, which captures the important features of the system and whose size is still comparable to $\varepsilon$, and a non-resonant part, whose size can be made exponentially small with respect to $\varepsilon^{-a}$, where the exponent $a>0$ depends only on the Diophantine exponent $\tau$. The result can also be extended to an arbitrary vector $\omega\in \R^n$, in which case the non-resonant part is exponentially small with respect to some function of $\varepsilon^{-1}$, this function depending only on $\Psi_\omega$ (see \cite{Bou12} for $d=n$). Such simpler systems are usually called resonant formal forms, and were first obtained in this context by Nekhoroshev (see \cite{Nek77}, or \cite{Pos93} for a nicer exposition) with the aim of establishing exponentially long (and global) stability estimates for the evolution of the action variables, under some geometrical assumptions on the integrable part. Apart from deriving such stability estimates, these resonant normal forms can also be used to prove the existence of invariant hyperbolic objects such as tori or cylinders of lower dimension. For hyperbolic tori (also called whiskered tori), which by definition possess stable and unstable manifolds (also known as whiskers), these normal forms can be used to estimate the ``angle" between these invariant manifolds at an intersection point: this is usually called the splitting of invariant manifolds (or ``separatrices"). 

The aim of this paper is to construct resonant normal forms with a small remainder in the broader class of Gevrey Hamiltonians, for an arbitrary frequency vector $\omega$. The case $d=1$ is known and our main theorem extends this result for any $1\leq d \leq n$. Our proof uses a method of periodic approximations: it is based on the simple case $d=1$, a Diophantine result on the approximation of a vector by independent periodic vectors proved in \cite{BF12} and a technique of composition of periodic averaging first used in \cite{BN09}. Note that our main result answers a question which was asked in \cite{LMS03}, concerning the ``interaction of Gevrey conditions with arithmetic properties in normal forms". In a second paper \cite{BouII12}, we will prove the corresponding result for finitely differentiable Hamiltonians using a similar method (but with fairly different technical details).

We give two applications of our main result. As a first application, we prove local and exponentially long stability estimates for the evolution of the action variables for perturbations of non-linear integrable systems, without any condition on the integrable part. We also prove global and exponentially long stability estimates for perturbations of integrable linear systems, and show that the latter result is essentially optimal. If one is interested in global stability estimates for perturbations of non-linear integrable systems, one needs to impose a geometric condition on the integrable part. It is known that for convex integrable systems, an essentially optimal result can be obtained using only the case $d=1$ (see \cite{Loc92} for analytic systems, \cite{MS02} for Gevrey systems and \cite{Bou10} for finitely differentiable systems). However, it seems that the general case $1\leq d \leq n$ is needed if one is interested in such global estimates for non-convex integrable Hamiltonians. In \cite{BN09} and \cite{Bou11}, special normal forms adapted to the problem were constructed using only periodic approximations: however, the quantitative results obtained there were very far from being optimal, and in particular our main theorem in this paper gives much more general and precise normal forms that should be useful in trying to also obtain essentially optimal results for perturbations of non-convex integrable systems.  As a second application, we will prove an exponentially small upper bound for the splitting of invariant manifolds for a hyperbolic torus.

\paraga Let us now state precisely our results, starting with the regularity assumption. Recall that $n \geq 2$ and $R>0$ have been fixed, and that our phase space is $\mathcal{D}_R=\T^n \times B_R$, where $B_R$ is the closed ball in $\R^n$, centered at the origin, of radius $R$ with respect to the supremum norm. 

Given $\alpha \geq 1$ and $L>0$, a real-valued function $f\in C^{\infty}(\mathcal{D}_R)$ is $(\alpha,L)$-Gevrey if, using the standard multi-index notation, we have
\begin{equation}\label{norm}
|f|_{G^{\alpha,L}(\mathcal{D}_R)}=\sup_{l\in \N^{2n}}|f|_{\alpha,L,l,R} < \infty, \quad |f|_{\alpha,L,l,R}=L^{|l|\alpha}(l!)^{-\alpha}|\partial^l f|_{C^{0}(\mathcal{D}_R)} 
\end{equation}
where $|\,.\,|_{C^{0}(\mathcal{D}_R)}$ is the usual supremum norm for functions on $\mathcal{D}_R$. The space of such functions, with the above norm, is a Banach space that we denote by $G^{\alpha,L}(\mathcal{D}_R)$. In the sequel, we shall forget the dependence on the domain in the notation and simply write $|\,.\,|_{\alpha,L}=|\,.\,|_{G^{\alpha,L}(\mathcal{D}_R)}$ and $|\,.\,|_{\alpha,L,l}=|\,.\,|_{\alpha,L,l,R}$. Next, for $p\geq 1$, a vector-valued function $F\in C^{\infty}(\mathcal{D}_R,\R^p)$ is $(\alpha,L)$-Gevrey if $F=(f_1,\dots,f_p)$ with $f_i \in G^{\alpha,L}(\mathcal{D}_R)$ for $1\leq i \leq p$. The space of such functions, that we denote by $G^{\alpha,L}(\mathcal{D}_R,\R^p)$, is also a Banach space with the norm
\[ |F|_{G^{\alpha,L}(\mathcal{D}_R,\R^p)}=\max_{1\leq i \leq p}|f_i|_{G^{\alpha,L}(\mathcal{D}_R)}.\]  
Similarly, we shall simply write $|\,.\,|_{\alpha,L}=|\,.\,|_{G^{\alpha,L}(\mathcal{D}_R,\R^p)}$. By identifying functions with values in $\T$ with $1$-periodic real-valued functions, we can also define $G^{\alpha,L}(\mathcal{D}_R,\T^n \times \R^n)$ and $|\,.\,|_{\alpha,L}=|\,.\,|_{G^{\alpha,L}(\mathcal{D}_R,\T^n \times \R^n)}$. Finally, we will also use the notation $G^{\alpha}=\bigcup_{L>0}G^{\alpha,L}$ to denote the space of functions which are $(\alpha,L)$-Gevrey for some $L>0$, on appropriate domains.

Note that analytic functions are a particular case of Gevrey functions, as one can check that $G^{1,L}(\mathcal{D}_R)$ is exactly the space of bounded real-analytic functions on $\mathcal{D}_R$ which extend as bounded holomorphic functions on the complex domain
\[ V_L(\mathcal{D}_{R})=\{(\theta,I)\in(\C^n/\Z^n)\times \C^{n} \; | \; |\mathcal{I}(\theta)|<L,\;d(I,B_R)<L\}, \] 
where $\mathcal{I}(\theta)$ is the imaginary part of $\theta$, $|\,.\,|$ the supremum norm on $\C^n$ and $d$ the associated distance on $\C^n$.

Since we will be only interested in local properties, it will be enough to consider first linear integrable Hamiltonians, that is $h(I)=l_\omega(I)=\omega\cdot I$. As we will see, our result for an arbitrary non-linear integrable Hamiltonian $h$ will be obtained from the linear case by a straightforward localization procedure. Therefore we shall first consider a Hamiltonian $H \in G^{\alpha,L}(\mathcal{D}_R)$ of the form
\begin{equation}\label{Hlin}
\begin{cases} \tag{G1}
H(\theta,I)=l_\omega(I)+f(\theta,I), \quad (\theta,I)\in \mathcal{D}_R, \\
|f|_{\alpha,L} \leq \varepsilon. 
\end{cases}
\end{equation}
We denote by $\{\,.\,,\,.\,\}$ the Poisson bracket associated to the symplectic structure on $\mathcal{D}_R$. For any vector $w \in \R^n$, let $X_w^t$ be the Hamiltonian flow of the linear integrable Hamiltonian $l_w(I)=w\cdot I$, and given any $g\in C^{1}(\mathcal{D}_R)$, we define the average (along the linear flow of frequency $w$) of $g$ by
\begin{equation}\label{ave}
[g]_w=\lim_{s\rightarrow +\infty}\frac{1}{s}\int_{0}^{s}g\circ X_{w}^{t}dt.
\end{equation}
Note that $\{g,l_w\}=0$ if and only if $g \circ X^t_w=g$ if and only if $g=[g]_w$. Recall that the function $\Psi_\omega$ has been defined in~\eqref{psi}, then we define the functions
\[ \Delta_\omega(Q)=Q\Psi_\omega(Q), \; Q \geq Q_\omega, \quad \Delta_{\omega}^*(x)=\sup\{Q \geq Q_\omega\; | \; \Delta_{\omega}(Q)\leq x\}, \; x \geq \Delta_\omega(Q_\omega). \]
Our first result is the following.

\begin{theorem}\label{thmlin}
Let $H$ be as in~\eqref{Hlin}. There exist positive constants $c$, $c_1$, $c_2$, $c_3$, $c_4$, $c_5$ and $C$ that depend only on $n,R,\omega,\alpha$ and $L$ such that if 
\begin{equation}\label{thr1}
\Delta_{\omega}^*(c\varepsilon^{-1}) \geq c_1,
\end{equation}
then there exists a symplectic map $\Phi \in G^{\alpha,\tilde{L}}(\mathcal{D}_{R/2},\mathcal{D}_R)$, where $\tilde{L}=CL$, such that 
\[ H \circ \Phi = l_\omega +[f]_\omega+ g + \tilde{f}, \quad \{g,l_\omega\}=0  \]
with the estimates 
\begin{equation}\label{estlindist}
|\Phi-\mathrm{Id}|_{\alpha,\tilde{L}} \leq c_2 \Delta_{\omega}^*(c\varepsilon^{-1})^{-1}
\end{equation} 
and
\begin{equation}\label{estlin}
|g|_{\alpha,\tilde{L}}\leq c_3 \varepsilon\Delta_{\omega}^*(c\varepsilon^{-1})^{-1}, \quad |\tilde{f}|_{\alpha,\tilde{L}}\leq c_4 \varepsilon \exp\left(-c_5 \Delta_{\omega}^*(c\varepsilon^{-1})^{\frac{1}{\alpha}}\right).
\end{equation}
\end{theorem}

On the domain $\mathcal{D}_{R/2}$, the above theorem states the existence of a symplectic Gevrey conjugacy, close to identity, between the original Hamiltonian and a Hamiltonian which is the sum of the integrable part, the average of the perturbation whose size is of order $\varepsilon$, a resonant part which by definition Poisson commutes with the integrable part and whose size is of order $\varepsilon(\Delta_{\omega}^*(c\varepsilon^{-1}))^{-1}$, and a general part whose size is now exponentially small with respect to $\Delta_{\omega}^*(c\varepsilon^{-1})^{\frac{1}{\alpha}}$. The first terms of this Hamiltonian, namely $l_\omega+[f]_\omega+g$, is what is called a resonant normal form, and the last term $\tilde{f}$ is a ``small" remainder.

Concerning the size of this remainder, in \cite{Bou12} an example in the analytic case for $d=n$ is given to show that this result is ``essentially" optimal: we will show below in \S\ref{s2} that this example can be easily adapted to the Gevrey case and for any $1\leq d \leq n$, and that our estimate for the remainder in the Gevrey case is also ``essentially" optimal. The word ``essentially" should be understood as follows: given $\omega$, one can always construct a sequence of positive real numbers $\varepsilon_j=\varepsilon_j(\omega)$, going to zero as $j$ goes to infinity, and a sequence of $\varepsilon_j$-perturbations $f_j$, such that the estimate for the remainder $\tilde{f}_j$ cannot be improved (see the next section for a precise statement). This, of course, does not preclude this estimate to be improved for other values of $\varepsilon$.

Now in the Diophantine case, the estimates of Theorem~\ref{thmlin} can be made more explicit. Indeed, in this case we have the upper bound $\Psi_\omega (Q)\leq \gamma^{-1}Q^{\tau}$ which gives the lower bound $\Delta_{\omega}^*(c\varepsilon^{-1}) \geq (c\gamma\varepsilon^{-1})^{\frac{1}{1+\tau}}$. The following corollary is then straightforward.

\begin{corollary}\label{thmlinD}
Let $H$ be as in~\eqref{Hlin}, and $\omega \in \Omega_d(\gamma,\tau)$. There exist positive constants $c$, $c_1$, $c_2$, $c_3$, $c_4$, $c_5$ and $C$ that depend only on $n,R,\omega,\alpha$ and $L$ such that if 
\begin{equation*}
\varepsilon \leq cc_1^{-(1+\tau)}\gamma,
\end{equation*}
then there exists a symplectic map $\Phi \in G^{\alpha,\tilde{L}}(\mathcal{D}_{R/2},\mathcal{D}_R)$, where $\tilde{L}=CL$, such that 
\[ H \circ \Phi = l_\omega +[f]_\omega+ g + \tilde{f}, \quad \{g,l_\omega\}=0  \]
with the estimates 
\begin{equation*}
|\Phi-\mathrm{Id}|_{\alpha,\tilde{L}} \leq c_2 (c^{-1}\gamma^{-1}\varepsilon)^{\frac{1}{1+\tau}}
\end{equation*} 
and
\begin{equation*}
|g|_{\alpha,\tilde{L}}\leq c_3 \varepsilon(c^{-1}\gamma^{-1}\varepsilon)^{\frac{1}{1+\tau}}, \quad |\tilde{f}|_{\alpha,\tilde{L}}\leq c_4 \varepsilon \exp\left(-c_5 (c\gamma\varepsilon^{-1})^{\frac{1}{\alpha(1+\tau)}}\right).
\end{equation*}
\end{corollary}

Even though the size of the remainder is essentially optimal as we already explained, in the Diophantine case we believe that the other estimates, concerning the size of the transformation $\Phi$ and the resonant term $g$, can be improved. As a matter of fact, these improvements are known in the analytic case for $d=n$ (see the comments at the end of this section). 

\paraga We can now state a local result for a perturbation of a general non-linear integrable Hamiltonian system, that is for a Hamiltonian $H \in G^{\alpha,L}(\mathcal{D}_R)$ of the form
\begin{equation}\label{Hnonlin}
\begin{cases} \tag{G2}
H(\theta,I)=h(I)+f(\theta,I), \quad (\theta,I)\in \mathcal{D}_R, \\
\nabla h(0)=\omega, \quad |h|_{\alpha,L} \leq 1, \quad |f|_{\alpha,L} \leq \varepsilon. 
\end{cases}
\end{equation} 
For a ``small" parameter $r>0$, we will focus on the domain $\mathcal{D}_r=\T^n \times B_r$, which is a neighbourhood of size $r$ of the unperturbed torus $\mathcal{T}_0=\T^n \times \{0\}$. 

Since we are interested in $r$-dependent domains in the space of action, the estimates for the derivatives with respect to the actions will have different size than the one for the derivatives with respect to the angles. To distinguish between them, we will split multi-integers $l\in \N^{2n}$ as $l=(l_1,l_2)\in \N^n \times \N^n$ so that $\partial^l =\partial_\theta^{l_1}\partial_I^{l_2}$ and $|l|=|l_1|+|l_2|$. Let us denote by $\mathrm{Id}_I$ and $\mathrm{Id}_\theta$ the identity map in respectively the action and angle space, and for a function $F$ with values in $\T^n \times \R^n$, we shall write $F=(F_\theta,F_I)$.

\begin{theorem}\label{thmnonlin}
Let $H$ be as in~\eqref{Hnonlin}. There exist positive constants $c$, $c_6$, $c_7$, $c_8$, $c_{9}$ and $c_{10}$ that depend only on $n,R,\omega,\alpha$ and $L$, such that if 
\begin{equation}\label{thr2}
\sqrt{\varepsilon} \leq r \leq R, \quad \Delta_{\omega}^*(cr^{-1}) \geq c_6,
\end{equation}
then there exists a symplectic map $\Phi \in G^{\alpha}(\mathcal{D}_{r/2},\mathcal{D}_r)$ such that 
\[ H \circ \Phi = h + g +[f]_\omega+ \tilde{f}, \quad \{g,l_\omega\}=0.  \]
Moreover, for $\tilde{L}=C L$ and any $l=(l_1,l_2)\in \N^{2n}$, we have the estimates 
\begin{equation}\label{estnonlindist}
|\Phi_I-\mathrm{Id}_I|_{\alpha,\tilde{L},l} \leq c_7 rr^{-|l_2|}\Delta_{\omega}^*(cr^{-1})^{-1}, \quad |\Phi_\theta-\mathrm{Id}_\theta|_{\alpha,\tilde{L},l} \leq c_7 r^{-|l_2|}\Delta_{\omega}^*(cr^{-1})^{-1}
\end{equation}
and
\begin{equation}\label{estnonlin}
|g|_{\alpha,\tilde{L},l}\leq c_8 r^2r^{-|l_2|}\Delta_{\omega}^*(cr^{-1})^{-1}, \quad |\tilde{f}|_{\alpha,\tilde{L},l}\leq c_{9} r^2r^{-|l_2|}\exp\left(-c_{10} \Delta_{\omega}^*(cr^{-1})^{\frac{1}{\alpha}}\right).
\end{equation}
\end{theorem}

The proof of Theorem~\ref{thmnonlin} is an easy consequence of Theorem~\ref{thmlin}: the Hamiltonian~\eqref{Hnonlin} on the domain $\mathcal{D}_r$ is, after scaling, equivalent to the Hamiltonian~\eqref{Hlin} on the domain $\mathcal{D}_1$, and due to the non-linearity, the small parameter is here naturally $r\geq\sqrt{\varepsilon}$ instead of $\varepsilon$. Note that no assumptions on $h$ is required, since no information on the vector $\nabla h(I)$ for $I \in B_r$ is used except of course at $I=0$. 

Let us note that the transformation $\Phi$, and the functions $g$ and $\tilde{f}$ are not $(\alpha,\tilde{L})$-Gevrey, because of the factors $r^{|l_2|}$ appearing in the estimates (nevertheless, they are still $(\alpha,\tilde{L}_r)$-Gevrey for some $\tilde{L}_r$ depending on $\tilde{L}$ and $r$, but this will not be used in the sequel). Even though these factors come from our method of proof (which consists in scaling the system so that it reduces to a perturbation of a linear integrable Hamiltonian), we believe this is not an artefact (similar factors appear in the analytic case when one uses Cauchy estimates to control the norm of the derivatives of a function from the norm of the function on a domain of size $r$). 

Let us also note that in the statement of Theorem~\ref{thmnonlin}, one has the freedom to choose any $r$ such that $ \sqrt{\varepsilon} \leq r \leq R$, provided $r$ is sufficiently small so that the second part of \eqref{thr2} is satisfied. In the sequel, we shall use the statement only for $r=\sqrt{\varepsilon}$ which appears to be the most natural value, but for some other purposes, for instance if one is interested in global stability estimates for non-linear systems (as in the Nekhoroshev theorem), it is useful (and perhaps necessary) to use the statement for bigger values of $r$.  

Now as before, in the Diophantine case we can give a more concrete statement.

\begin{corollary}\label{thmnonlinD}
Let $H$ be as in~\eqref{Hnonlin}, and $\omega \in \Omega_d(\gamma,\tau)$. There exist positive constants $c$, $c_6$, $c_7$, $c_8$, $c_{9}$ and $c_{10}$ that depend only on $n,R,\omega,\alpha$ and $L$, such that if 
\begin{equation}\label{thr3}
\sqrt{\varepsilon} \leq r \leq R, \quad r \leq cc_6^{-(1+\tau)}\gamma,
\end{equation}
then there exists a symplectic map $\Phi \in G^{\alpha}(\mathcal{D}_{r/2},\mathcal{D}_r)$ such that 
\[ H \circ \Phi = h + g +[f]_\omega+ \tilde{f}, \quad \{g,l_\omega\}=0.  \]
Moreover, for $\tilde{L}=C L$ and any $l=(l_1,l_2)\in \N^{2n}$, we have the estimates 
\begin{equation*}
|\Phi_I-\mathrm{Id}_I|_{\alpha,\tilde{L},l} \leq c_7 rr^{-|l_2|}(c^{-1}\gamma^{-1}r)^{\frac{1}{1+\tau}}, \quad |\Phi_\theta-\mathrm{Id}_\theta|_{\alpha,\tilde{L},l} \leq c_7 r^{-|l_2|}(c^{-1}\gamma^{-1}r)^{\frac{1}{1+\tau}}
\end{equation*}
and
\begin{equation*}
|g|_{\alpha,\tilde{L},l}\leq c_8 r^2r^{-|l_2|}(c^{-1}\gamma^{-1}r)^{\frac{1}{1+\tau}}, \quad |\tilde{f}|_{\alpha,\tilde{L},l}\leq c_{9} r^2r^{-|l_2|}\exp\left(-c_{10} (c\gamma r^{-1})^{\frac{1}{\alpha(1+\tau)}}\right).
\end{equation*}
\end{corollary}

\paraga It is perhaps worthwhile to compare our results with previous results, which were restricted mostly to the analytic case $\alpha=1$ or the periodic case $d=1$.

In the analytic case $\alpha=1$, results similar to Theorem~\ref{thmlin} (and also to Theorem~\ref{thmnonlin}) are contained in \cite{Pos93} and \cite{DG96} (see also \cite{Sim94}). Note that our results are however more accurate than those contained in \cite{Pos93}: in the estimates for $\Phi$ and $g$, the term $\Delta_{\omega}^*(c\varepsilon^{-1})^{-1}$ we obtain here is replaced by $1$. This does not lead to any improvement if one is interested in global stability estimates for perturbations of non-linear integrable Hamiltonians because the geometry of resonances prevails in this case, but this improvement on the estimate for $\Phi$ is visible for global stability estimates for perturbations of linear integrable Hamiltonians (compare Theorem~\ref{stablin} below with Theorem $5$ of \cite{Pos93}) and for local stability estimates for perturbations of non-linear integrable Hamiltonians. Also, the improvement on the estimate for $g$ is important for the application to the splitting of invariant manifolds (see Remark~\ref{rem}). Now concerning Corollary~\ref{thmlinD} in the special case $\alpha=1$, one can in fact find better results in \cite{Fas90}: still in the estimates for $\Phi$ and $g$, the term $(c^{-1}\gamma^{-1}\varepsilon)^{\frac{1}{1+\tau}}$ we find is replaced by $c^{-1}\gamma^{-1}\varepsilon$. This discrepancy can be explained as follows. The perturbation theory of quasi-periodic motions essentially relies on solving a certain equation (called homological) which consists of integrating a function along the linear flow of frequency $\omega$. Assuming that the function has zero average along the flow, a formal solution always exists, and it is a smooth convergent solution if and only if $\omega$ is Diophantine. So if one considers non-Diophantine vectors (which is important for several problems, such as Nekhoroshev type estimates), one can only solve an ``approximate" equation. The usual approach is to approximate the function by a nicer function (usually a polynomial function), whereas here, as in \cite{BF12}, we just approximate the frequency by a nicer frequency (a periodic frequency). But if we restrict to Diophantine vectors, then one can actually solve the exact equation, and this ultimately leads to better estimates.    

Now in the periodic case $d=1$, the result is known for any $\alpha \geq 1$ (see \cite{MS02}), and our proof of the general case $1 \leq d \leq n$ crucially relies on it. In fact, in \cite{MS02}, one has a more accurate result, namely the existence of a formal transformation to a formal normal form, both having Gevrey asymptotics. Such a result is stronger than the existence of a transformation to a normal form with an exponentially small remainder, as the latter is implied by the former (but not the contrary). In the analytic case $\alpha=1$, these asymptotics were first established in \cite{RS96}. 

Finally, there are some related results for $d=n$ and any $\alpha \geq 1$. First in \cite{MP10}, the dynamics in the neighbourhood of a Gevrey Lagrangian torus is considered, and a Gevrey normal form is constructed. The system studied there can be brought to a system of the form we considered, though with a special perturbation, and the result obtained in \cite{MP10} is therefore close but different. Then, in an unpublished manuscript \cite{Sau06}, a result analogous to \cite{MS02}, that is the construction of a formal transformation to a formal normal form with Gevrey asymptotics, is proved in the case of a Diophantine vector. As before, these two results give more information but in a different setting, moreover they are based on solving exact homological equations and so they cannot apply to non-Diophantine vectors. 

\paraga Finally let us describe the plan of this paper. In \S\ref{s2}, we will deduce from Theorem~\ref{thmlin} and Theorem~\ref{thmnonlin} exponentially large stability estimates for the evolution of the action variables, which are global for perturbations of linear integrable systems and only local for perturbations of non-linear integrable systems. In the first case, that is for perturbations of linear integrable systems, we will show on an example that these estimates cannot be improved in general. Then in \S\ref{s3}, we will use Theorem~\ref{thmnonlin} to prove a result of exponential smallness for the splitting of invariant manifolds of a hyperbolic tori. The proof of Theorem~\ref{thmlin} and Theorem~\ref{thmnonlin} will be given in \S\ref{s4}. Finally, we will gather in a appendix some technical estimates concerning Gevrey functions that are used to prove our theorems.

To simplify the notations and improve the readability, when convenient we shall replace constants depending on $n,R,\omega,\alpha$ and $L$ that can, but need not be, made explicit by a $\cdot$, that is an expression of the form $u \MP v$  means that there exists a constant $c>0$, that depends on the above set of parameters, such that $u \leq cv$. Similarly, we will use the notations $u \PS v$ and $u \EP v$.

\section{Application to stability estimates}\label{s2}

In this section, we will give direct consequences of our normal forms Theorem~\ref{thmlin} and Theorem~\ref{thmnonlin} to the stability of the action variables.

Recall that we have defined $F$ as the smallest rational subspace containing the vector $\omega$. The key point to obtain stability estimates is the following observation. By definition of the Poisson Bracket, the equalities $\{[f]_\omega,l_\omega\}=\{g,l_\omega\}=0$ imply that $l_\omega$ is a first integral of the normal form $l_\omega+[f]_\omega+g$ in the statement of Theorem~\ref{thmlin} (or a first integral of $h+[f]_\omega+g$, since $h$ is integrable, in the statement of Theorem~\ref{thmnonlin}). Therefore the action variables of all solutions of the normal form stay constant along $F$, that is $\Pi_F(I(t)-I_0)=0$ as long as they are defined, if $I_0=I(0)$ is an arbitrary initial action and where we have denoted by $\Pi_F$ the projection onto the subspace $F$. Now if we add a small remainder to the normal form, the quantities $|\Pi_F(I(t)-I_0)|$, as long as they are defined, will have small variations for an interval of time which is roughly the inverse of the size of (the partial derivative with respect to $\theta$ of) the remainder. Coming back to our original system, since it is conjugated by a symplectic map which is close to identity to such a normal form with a small remainder, the same property remains true as long as the solution stays on the domain where the conjugacy is defined.     

\paraga Let us state precisely the result, starting with a perturbation of a linear integrable Hamiltonian as in~\eqref{Hlin}.

\begin{theorem}\label{stablin}
Under the notations and assumptions of Theorem~\ref{thmlin}, let $(\theta(t),I(t))$ be a solution of the Hamiltonian system defined by $H$, with $(\theta_0,I_0)\in \mathcal{D}_{R/4}$ and let $T_0$ be the smallest $t \in ]0,+\infty]$ such that $(\theta(t),I(t))\notin \mathcal{D}_{R/4}$. Then we have the estimates
\[ |\Pi_F(I(t)-I_0)| \leq \delta, \quad |t|\MP \min \left\{T_0, \delta\varepsilon^{-1}\exp\left(\cdot\Delta_\omega^*(\cdot\varepsilon^{-1})^{\frac{1}{\alpha}}\right)\right\} \]
for any $(\Delta_\omega^*(\cdot\varepsilon^{-1}))^{-1} \MP \delta\MP 1$. Moreover, if $F=\R^n$, then 
\[ |I(t)-I_0| \leq \delta, \quad |t|\MP \delta\varepsilon^{-1}\exp\left(\cdot\Delta_\omega^*(\cdot\varepsilon^{-1})^{\frac{1}{\alpha}}\right). \]
\end{theorem}

\begin{corollary}\label{stablinD}
Under the notations and assumptions of Corollary~\ref{thmlinD}, let $(\theta(t),I(t))$ be a solution of the Hamiltonian system defined by $H$, with $(\theta_0,I_0)\in \mathcal{D}_{R/4}$ and let $T_0$ be the smallest $t \in ]0,+\infty]$ such that $(\theta(t),I(t))\notin \mathcal{D}_{R/4}$. Then we have the estimates
\[ |\Pi_F(I(t)-I_0)| \leq \delta, \quad |t|\MP \min \left\{T_0, \delta\varepsilon^{-1}\exp\left(\cdot(\gamma\varepsilon^{-1})^{\frac{1}{\alpha(1+\tau)}}\right)\right\} \]
for any $(\gamma^{-1}\varepsilon)^{\frac{1}{1+\tau}} \MP \delta\MP 1$. Moreover, if $F=\R^n$, then 
\[ |I(t)-I_0| \leq \delta, \quad |t|\MP \delta\varepsilon^{-1}\exp\left(\cdot(\gamma\varepsilon^{-1})^{\frac{1}{\alpha(1+\tau)}}\right). \]
\end{corollary}

Note that we have exponential stability for the action variables only in the case $F=\R^n$, that is when $\omega$ is non-resonant, since in this case we can ensure that $T_0$ is at least exponentially large. In general, the stability result is only partial since it involves the projection of the action components onto $F$, and the time $T_0$, which is easily seen to be always at least of order $\delta\varepsilon^{-1}$, cannot be much larger in general (it is easy to construct examples for which the action variables of a solution drift along the orthogonal complement of $F$ with a speed of order $\varepsilon$).

Note that in the particular case where $\delta \EP (\Delta_\omega^*(\cdot\varepsilon^{-1}))^{-1}$, we have 
\[ |\Pi_F(I(t)-I_0)| \MP (\Delta_\omega^*(\cdot\varepsilon^{-1}))^{-1}\]
for
\[|t|\MP \min\left\{T_0,(\Delta_\omega^*(\cdot\varepsilon^{-1}))^{-1}\varepsilon^{-1}\exp\left(\cdot\Delta_\omega^*(\cdot\varepsilon^{-1})^{\frac{1}{\alpha}}\right)\right\} \]
and in the Diophantine case, for $\delta \EP (\gamma^{-1}\varepsilon)^{\frac{1}{1+\tau}}$, we have
\[ |\Pi_F(I(t)-I_0)| \MP (\gamma^{-1}\varepsilon)^{\frac{1}{1+\tau}}\]
for
\[|t|\MP \min\left\{T_0,\cdot\gamma^{-\frac{1}{1+\tau}}\varepsilon^{-\frac{\tau}{1+\tau}}\exp\left(\cdot(\gamma\varepsilon^{-1})^{\frac{1}{\alpha(1+\tau)}}\right)\right\}. \]
Still in the Diophantine case, we expect that it could be possible to take $\delta \EP \gamma^{-1}\varepsilon$ (see the comments at the end of the first section). 

The deduction of Theorem~\ref{estlin} from Theorem~\ref{thmlin} is very classical, but for completeness we give the details.

\begin{proof}
Using Theorem~\ref{thmlin}, it is enough to prove the statement for a solution $(\theta(t),I(t))$ of the Hamiltonian $\tilde{H}=H \circ \Phi$, with $(\theta_0,I_0)\in \mathcal{D}_{R/2}$ and $T_0$ being the smallest $t \in ]0,+\infty]$ such that $(\theta(t),I(t))\notin \mathcal{D}_{R/2}$. Indeed, we have $\tilde{H}=H \circ \Phi$, where $\Phi : \mathcal{D}_{R/2} \rightarrow \mathcal{D}_{R}$ is symplectic. Moreover, from~\eqref{estlindist}, we have the estimate
\[ |\Phi-\mathrm{Id}|_{C^0(\mathcal{D}_{R/2})} \leq |\Phi-\mathrm{Id}|_{\alpha,\tilde{L}}  \MP (\Delta_\omega^*(\cdot\varepsilon^{-1}))^{-1}.\]
so we can ensure that the image of $\Phi$ contains the domain $\mathcal{D}_{R/4}$, and since we have assumed that $(\Delta_\omega^*(\cdot\varepsilon^{-1}))^{-1} \MP \delta\MP 1$, the statement for $\tilde{H}$ trivially implies the statement for $H$ (only with different implicit constants). Note that in fact we will not use the first part of~\eqref{estlin}, but only the following obvious consequence of the second part of~\eqref{estlin}:
\begin{equation}\label{estestlin2}
|\partial_\theta \tilde{f}|_{C^0(\mathcal{D}_{R/2})} \MP |\tilde{f}|_{\alpha,\tilde{L}}\MP \varepsilon \exp\left(-\cdot (\Delta_\omega^*(\cdot\varepsilon^{-1}))^{\frac{1}{\alpha}}\right).
\end{equation}  
Recall that the equation of motions of $\tilde{H}$ are
\[ \dot{\theta}=\partial_I \tilde{H}(\theta,I), \quad \dot{I}=-\partial_{\theta} \tilde{H}(\theta,I)  \]
and using the mean value theorem and the fact that $l_\omega$ is integrable,
\[ |I(t)-I_0| \leq |t||\partial_\theta\tilde{H}|_{C^0(\mathcal{D}_{R/2})} \leq |t||\partial_\theta ([f]_\omega+g+\tilde{f})|_{C^0(\mathcal{D}_{R/2})}, \quad |t|\leq T_0.  \]
But $\{g,l_\omega\}=0$ is equivalent to $\partial_\theta g(\theta,I) \in F^\perp$, and similarly for $[f]_\omega$, hence if we project onto $F$, we have
\[ |\Pi_{F}(I(t)-I_0)| \leq |t||\partial_\theta \tilde{f}|_{C^0(\mathcal{D}_{R/2})}, \quad |t|\leq T_0.  \]
Now using~\eqref{estestlin2} and 
\[ |t|\MP \delta\varepsilon^{-1}\exp\left(\cdot (\Delta_\omega^*(\cdot\varepsilon^{-1}))^{\frac{1}{\alpha}}\right) \]
we obtain
\[ |\Pi_F(I(t)-I_0)| \leq \delta, \quad |t|\MP \min \left\{T_0, \delta\varepsilon^{-1}\exp\left(\cdot (\Delta_\omega^*(\cdot\varepsilon^{-1}))^{\frac{1}{\alpha}}\right)\right\} \]
for any $(\Delta_\omega^*(\cdot\varepsilon^{-1}))^{-1} \MP \delta\MP 1$. Finally, note that if $F=\R^n$, the map $\Pi_F$ is the identity so that
\[ T_0\PS \delta\varepsilon^{-1}\exp\left(-\cdot (\Delta_\omega^*(\cdot\varepsilon^{-1}))^{\frac{1}{\alpha}}\right) \]
and hence
\[ |I(t)-I_0| \leq \delta, \quad |t|\MP \delta\varepsilon^{-1}\exp\left(-\cdot (\Delta_\omega^*(\cdot\varepsilon^{-1}))^{\frac{1}{\alpha}}\right). \]
This concludes the proof.
\end{proof}

\paraga Now we will show on an example that the estimates of Theorem~\ref{stablin}, and therefore the estimate on the remainder in Theorem~\ref{thmlin}, are essentially optimal. 

\begin{theorem}\label{explin}
Let $\omega\in \R^n\setminus\{0\}$. Then there exist a sequence $(\varepsilon_j)_{j\in\N}$ of positive real numbers and a sequence $(f_j)_{j\in\N}$ of functions in $G^{\alpha,L}(\mathcal{D}_R)$, with 
\[ \varepsilon_j \MP |f_j|_{\alpha,L} \MP \varepsilon_j, \quad \lim_{j \rightarrow +\infty}\varepsilon_j=0,\] 
such that for $j \PS 1$, the Hamiltonian system defined by $H_j=l_\omega+f_j$ has solutions $(\theta(t),I(t))$ which satisfy
\[ |t|\varepsilon_j\exp\left(-\cdot (\Delta_\omega^*(\cdot\varepsilon_j^{-1}))^{\frac{1}{\alpha}}\right) \MP |\Pi_F(I(t)-I_0)|\MP|t|\varepsilon_j\exp\left(-\cdot (\Delta_\omega^*(\cdot\varepsilon_j^{-1}))^{\frac{1}{\alpha}}\right). \]
\end{theorem}

\begin{proof}
First recall the following fact from linear algebra: there exists an invertible square matrix $A$ of size $n$ with integer coefficients such that $\omega=A(\varpi,0) \in \R^d \times \R^m=\R^n$ with $\varpi \in \R^d$ non-resonant. So replacing $H$ by $H\circ \Phi_A$, where $\Phi_A: \T^n \times \R^n \rightarrow \T^n \times \R^n$ is the linear symplectic map given by $\Phi_A(\theta,I)=(A\theta,(A^{-1})^{t}I)$, we can assume that $\omega=(\varpi,0)$ as the statement for $H$ and $H\circ \Phi_A$ are equivalent (up to constants depending on $A$, and so on $n$ and $\omega$). Moreover, dividing $\varpi$ by $\pm|\varpi|$ and reordering its components if necessary, we can also assume that $\varpi=(1,\varpi,\dots,\varpi_{d-1})$, with $|\varpi_i|\leq 1$ for $1\leq i \leq d-1$.

So we have reduced the general case to 
\[ \omega=(1,\varpi_1, \dots, \varpi_{d-1},0)\in \R^d \times \R^m=\R^n, \quad F=\R^d \times \{0\}.\]
Let us denote by $(p_j/q_j)_{j\in\N}$ the sequence of the convergents of $\varpi_1$ for instance. We have the classical inequalities
\begin{equation*}
(q_j+q_{j+1})^{-1}<|q_j\varpi_1-p_j|<q_{j+1}^{-1}, \quad j\in\N, 
\end{equation*}
and since $q_{j+1}>q_j$, this gives
\begin{equation}\label{est1}
(2q_{j+1})^{-1}<|q_j\varpi_1-p_j|<q_{j+1}^{-1}, \quad j\in\N. 
\end{equation}
Now by definition of $\Psi_\omega$, we obtain
\begin{equation}\label{est2}
q_{j+1}<\Psi_\omega(q_j)<2q_{j+1}.
\end{equation}
The perturbation $f_j$ will be of the form
\[ f_j(\theta,I)=f_j^1(I)+f_j^2(\theta), \quad (\theta,I)\in \mathcal{D}_R. \]
First, we choose $f_j^1(I)=v_j\cdot I-\omega \cdot I$, where $v_j=(1,p_j/q_j,\varpi_2,\dots,\varpi_{d-1},0)$. We set
\[ \varepsilon_j=|\varpi_1-p_j/q_j|. \]
From the inequalities~\eqref{est1} and \eqref{est2}, recalling the definitions of $\Delta_\omega$ and $\Delta^*_\omega$, we have 
\begin{equation}\label{est4}
(\Delta_\omega(q_j))^{-1} \MP \varepsilon_j \MP (\Delta_\omega(q_j))^{-1}, \quad \Delta_\omega^*(\cdot \varepsilon_j^{-1}) < q_j < \Delta_\omega^*(\cdot \varepsilon_j^{-1}).
\end{equation}
Now by definition, $|f_j^1|_{\alpha,L}\EP \varepsilon_j$. Then, if we let $k_j=(p_j,-q_j,0,\dots,0)\in \Z^n$, we define $f_j^2(\theta)=\varepsilon_j\mu_j \sin(2\pi k_j\cdot\theta)$
with $\mu_j$ to be chosen. Note that $|k_j|=q_j$ since $|q_j|\geq|p_j|$ (as $|\varpi_1|\leq 1$). It is easy to estimate
\[ \varepsilon_j\mu_j \exp(\cdot |k_j|^{1/\alpha}) \MP |f_j^2|_{\alpha,L}\MP \varepsilon_j\mu_j \exp(\cdot |k_j|^{1/\alpha})   \]
and so 
\[ \varepsilon_j\mu_j \exp(\cdot q_j^{1/\alpha}) \MP |f_j^2|_{\alpha,L}\MP \varepsilon_j\mu_j \exp(\cdot q_j^{1/\alpha}). \]
Therefore we choose $\mu_j \EP \exp(-\cdot q_j^{1/\alpha})$ so that $\varepsilon_j \MP |f_j^2|_{\alpha,L} \MP \varepsilon_j$. Finally, $f_j=f^1_j+f^2_j$ so $\varepsilon_j \MP |f_j|_{\alpha,L} \MP \varepsilon_j$, and $\varepsilon_j \rightarrow 0$ when $j\rightarrow +\infty$. Now we can write the Hamiltonian
\begin{eqnarray*}
H_j(\theta,I) & = & l_\omega(I)+f_j(\theta,I) \\
& = & \omega\cdot I+v_j\cdot I-\omega \cdot I+\varepsilon_j\mu_j \sin(2\pi k_j\cdot\theta) \\
& = & v_j\cdot I+\varepsilon_j\mu_j \sin(2\pi k_j\cdot\theta) 
\end{eqnarray*}
and as $k_j.v_j=0$, the associated system is easily integrated:
\begin{equation*} 
\left\{  
\begin{array}{ccl}
\theta(t) & = & \theta_0+tv_j \quad [\Z^n]\\
 I(t) &= &I_0-t2\pi k_j\varepsilon_j\mu_j\cos(2\pi k_j.\theta_0).
\end{array}
\right. 
\end{equation*}
Choosing any solution with initial condition $(\theta_0,I_0)$ satisfying $k_j.\theta_0=0$, $\cos(2\pi k_j.\theta_0)=1$ and using the fact that $|k_j|=q_j$ and $\mu_j \EP \exp(-\cdot q_j^{1/\alpha})$, we obtain
\[ |I(t)-I_0|\EP|t|\varepsilon_jq_j\exp(-\cdot q_j^{1/\alpha}). \]
For $j\PS 1$, $q_j\PS 1$ and we have $\exp(-\cdot q_j^{1/\alpha}) < q_j\exp(-\cdot q_j^{1/\alpha}) < \exp(-\cdot q_j^{1/\alpha})$ for well-chosen implicit constants, hence
\[ |t|\varepsilon_j\exp(-\cdot q_j^{1/\alpha})\MP |I(t)-I_0|\MP|t|\varepsilon_j\exp(-\cdot q_j^{1/\alpha}). \]
Using~\eqref{est4} this gives
\[ |t|\varepsilon_j\exp\left(-\cdot(\Delta_\omega^*(\cdot \varepsilon_j^{-1}))^{\frac{1}{\alpha}}\right) \MP |I(t)-I_0| \MP |t|\varepsilon_j\exp\left(-\cdot(\Delta_\omega^*(\cdot \varepsilon_j^{-1}))^{\frac{1}{\alpha}}\right) \]
and as the vector $k_j$ belongs to $F=\R^d \times \{0\}$, the solution we have constructed satisfy $\Pi_F(I(t)-I_0)=I(t)-I_0$. This concludes the proof.
\end{proof}

\paraga For a perturbation of a non-linear integrable Hamiltonian as in~\eqref{Hnonlin}, we have a stability result similar to Theorem~\ref{stablin}. 

\begin{theorem}\label{stabnonlin}
Under the notations and assumptions of Theorem~\ref{thmnonlin}, let $(\theta(t),I(t))$ be a solution of the Hamiltonian system defined by $H$, with $(\theta_0,I_0)\in \mathcal{D}_{r/4}$ and let $T_0$ be the smallest $t \in ]0,+\infty]$ such that $(\theta(t),I(t))\notin \mathcal{D}_{r/4}$. Then we have the estimates
\[ |\Pi_F(I(t)-I_0)| \leq \delta, \quad |t|\MP \min \left\{T_0, \delta r^{-2}\exp\left(\cdot\Delta_\omega^*(\cdot r^{-1})^{\frac{1}{\alpha}}\right)\right\} \]
for any $r(\Delta_\omega^*(\cdot r^{-1}))^{-1} \MP \delta\MP r$. Moreover, if $F=\R^n$, then 
\[ |I(t)-I_0| \leq \delta, \quad |t|\MP \delta r^{-2}\exp\left(\cdot\Delta_\omega^*(\cdot r^{-1})^{\frac{1}{\alpha}}\right). \]
\end{theorem}

\begin{corollary}\label{stabnonlinD}
Under the notations and assumptions of Corollary~\ref{thmnonlinD}, let $(\theta(t),I(t))$ be a solution of the Hamiltonian system defined by $H$, with $(\theta_0,I_0)\in \mathcal{D}_{r/4}$ and let $T_0$ be the smallest $t \in ]0,+\infty]$ such that $(\theta(t),I(t))\notin \mathcal{D}_{r/4}$. Then we have the estimates
\[ |\Pi_F(I(t)-I_0)| \MP \delta, \quad |t|\MP \min \left\{T_0, \delta r^{-2}\exp\left(\cdot(\gamma r^{-1})^{\frac{1}{\alpha(1+\tau)}}\right)\right\} \]
for any $r(\gamma^{-1}r)^{\frac{1}{1+\tau}} \MP \delta\MP r$. Moreover, if $F=\R^n$, then 
\[ |I(t)-I_0| \MP \delta, \quad |t|\MP \delta r^{-2}\exp\left(\cdot(\gamma r^{-1})^{\frac{1}{\alpha(1+\tau)}}\right). \]
\end{corollary}

The proof of Theorem~\ref{stabnonlin} is entirely similar to the proof of Theorem~\ref{stablin}, replacing the use of Theorem~\ref{thmlin} by the use of Theorem~\ref{thmnonlin}, since the only information we need to derive these estimates from the resonant normal form with a remainder are $C^0$ estimates for the distance to the identity (with respect to the action variables) of the conjugacy $\Phi$ and $C^0$ estimates for the partial derivatives (with respect to the angle variables) of the remainder $\tilde{f}$, and these information are contained in Theorem~\ref{thmnonlin}. Therefore we shall not repeat the details. 

These local estimates of stability (or variants of them) are at the basis of global (in phase space) estimates of stability for perturbations of non-linear integrable systems, provided that the integrable part satisfy some geometric assumption. However, we don't know if it is possible to construct an example such as in Theorem~\ref{explin} that would show that the estimates of Theorem~\ref{stabnonlin} are ``essentially" optimal.

\section{Application to the splitting of invariant manifolds}\label{s3}

In this section, we apply our normal form results to a different but ultimately related problem, which is the so-called ``splitting" of invariant manifolds (or ``splitting" of ``separatrices"). Roughly speaking, if a Hamiltonian system $H$ as in \eqref{Hnonlin} has a suitable ``hyperbolic" invariant torus for which the stable and unstable manifolds intersect, the problem is to evaluate in some sense the ``angle" between these invariant manifolds at the intersection point. This is an important problem in itself, and this is also deeply connected to the problem of the speed of instability for the action variables (which is known as the speed of Arnold diffusion). 

The general principle is that the ``splitting" is exponentially small for analytic systems and the literature on the subject is huge. Here we shall closely follows \cite{LMS03}, Chapter \S $2$, where an approach to obtain exponentially small upper bounds in the analytic case is given based on normal forms techniques. The results contained in \cite{LMS03} are quite general, as they are valid for any number of degrees of freedom, any Diophantine frequency vector and without any restriction on the perturbation, assuming the torus exists and that its stable and unstable manifolds intersect. However, because of this great generality, these results are not very accurate, they only give an upper bound with a reasonable value for the exponent in the exponential factor. Much more accurate results, such as an asymptotic formula and reasonable values for the other constants involved, can be obtained in much more restricted situations (essentially for two degrees of freedom, specific frequency vectors and specific perturbations, see \cite{BFGS12} for some recent results and references). 

Our aim here is to generalize the results of \cite{LMS03} for Hamiltonian systems which are only Gevrey regular. We will also assume the existence of a ``hyperbolic" torus, together with the property that its invariant manifolds intersect (that is, the existence of homoclinic orbits), simply because conditions that ensure the existence of these objects are well-known (we will quote some of these conditions).  

\paraga We will not try to give an abstract definition of a hyperbolic torus for a Hamiltonian system, first because we will deal with rather concrete examples, but also because we could not find any satisfactory abstract definition (for instance one which would ensure the existence and uniqueness of stable and unstable manifolds, we refer to \cite{LMS03} and \cite{BT00} for some attempts).

Now let us consider the setting as described in~\eqref{Hnonlin}. Let $\mathcal{T}_0=\T^n \times \{0\}$ be the invariant torus, for the integrable system, with frequency $\omega\in \R^n \setminus\{0\}$. Without loss of generality (as we explained in the proof of Theorem~\ref{explin}), we may already assume that $\omega=(\varpi,0)\in \R^d\times \R^m=\R^n$ with $\varpi \in \R^d$ non-resonant.

If $\omega \in \Omega_n(\gamma,\tau)$, by KAM theory this torus will persist under any sufficiently small and regular perturbation, provided that $h$ is non-degenerate (in the analytic case, the assumption that $\omega \in \Omega_n(\gamma,\tau)$ can be weakened). This is not true if $d \leq n-1$; however, under appropriate assumptions on the system, one can still apply KAM theory to prove the existence of an invariant $d$-dimensional torus, with frequency $\varpi$, which is hyperbolic in the sense that it possesses stable and unstable manifolds. Moreover, such a tori will be isotropic and its asymptotic manifolds will be Lagrangian. If the stable and unstable manifolds intersect, following \cite{LMS03}, we can define a symmetric matrix of size $n$ at a given homoclinic point, called a splitting matrix, the eigenvalues of which are called the splitting angles. Our result is that there exists at least $d$ splitting angles which are exponentially small (see Theorem~\ref{thmsplit}). The proof of this result will be analogous to the one in \cite{LMS03}. However, in \cite{LMS03}, it seems that the normal forms they used, which are taken from \cite{Pos93}, are not accurate enough to derive the result they claimed (see Remark~\ref{rem}).

Before stating precisely the result, we need some rather lengthy preparations in order to introduce our definitions and assumptions (the validity of our assumptions will be briefly discussed after the statement of Theorem~\ref{thmsplit}). 

First let us split our angle-action coordinates $(\theta,I) \in \mathcal{D}_R=\T^n \times B_R$ accordingly to  $\omega=(\varpi,0)\in \R^d\times \R^m=\R^n$: we write $\theta=(\theta_1,\theta_2) \in \T^d \times \T^m$ and $I=(I_1,I_2) \in B_R^d \times B_R^m$, where $B_R^d=B_R \cap \R^d$ and $B_R^m=B_R \cap \R^m$. We shall always assume here that $m\geq 1$, that is $d\leq n-1$. Then Theorem~\ref{thmnonlin} states that for the value $r=2\sqrt{\varepsilon}$, if 
\[\Delta_{\omega}^*(\cdot\sqrt{\varepsilon}^{-1}) \PS 1,\] 
there exists a symplectic map $\Phi \in G^{\alpha}(\mathcal{D}_{r/2},\mathcal{D}_{r})$, satisfying the estimate~\eqref{estnonlindist}, such that
\[ H \circ \Phi = h +[f]_\omega+ g + \tilde{f}, \quad [f]_\omega(\theta_2,I)=\int_{\theta_1 \in \T^d}f(\theta,I)d\theta_1, \quad g(\theta,I)=g(\theta_2,I)  \]
with $g$ and $\tilde{f}$ satisfying the estimates~\eqref{estnonlin}. As in the proof of Theorem~\ref{thmnonlin}, it will be more convenient to use a rescaled version of the Hamiltonian $H \circ \Phi$. Consider the map
\[ \sigma : (\theta,I) \longmapsto (\theta,\sqrt{\varepsilon}I) \]
which sends the domain $\mathcal{D}_{1}$ onto $\mathcal{D}_{r/2}$, and let 
\[ \mathcal{H}=\varepsilon^{-1}(H\circ \Phi \circ \sigma)=\varepsilon^{-1} h\circ \sigma +\varepsilon^{-1}[f]_\omega \circ \sigma + \varepsilon^{-1} g \circ \sigma + \varepsilon^{-1} \tilde{f} \circ \sigma\] 
be the rescaled Hamiltonian, which is defined on $\mathcal{D}_{1}$. Note that in the proof of Theorem~\ref{thmnonlin}, we used the same scaling map but the scaling factor was $\sqrt{\varepsilon}^{-1}$ in order to bring the linear part of $h$ to order one, while here the scaling factor is $\varepsilon^{-1}$ since it is the quadratic part of $h$ we want to bring to order one. Consequently, the solutions of the Hamiltonian system defined by $\mathcal{H}$ coincide with those of $H\circ \Phi$ only after scaling time by $\sqrt{\varepsilon}$: since we will be only interested in invariant and asymptotic manifolds, it makes no difference to consider $\mathcal{H}$ instead of $H \circ \Phi$. Now let us define
\begin{equation}\label{mu}
 \lambda=\lambda(\varepsilon)=(\Delta_{\omega}^*(\cdot\sqrt{\varepsilon}^{-1}))^{-1}, \quad \mu=\mu(\varepsilon)=\exp(-\cdot\lambda(\varepsilon)^{-\frac{1}{\alpha}})=\exp(-\cdot\lambda^{-\frac{1}{\alpha}}). 
\end{equation} 
For any $l=(l_1,l_2)\in \N^{2n}$, we have
\[ \partial^l (g\circ \sigma)=\partial_\theta^{l_1}\partial_I^{l_2}(g\circ \sigma)=\sqrt{\varepsilon}^{|l_2|}(\partial^lg)\circ \sigma\] 
so
\[ |\partial^l (g\circ\sigma)|_{C^0(\mathcal{D}_{1})} \leq \sqrt{\varepsilon}^{|l_2|} |\partial^l g|_{C^0(\mathcal{D}_{r/2})} \]
and therefore
\[ |g\circ\sigma|_{\alpha,\tilde{L},l} \leq \sqrt{\varepsilon}^{|l_2|} |g|_{\alpha,\tilde{L},l}.\]
Now using~\eqref{estnonlin}, this gives
\[ |g\circ\sigma|_{\alpha,\tilde{L},l} \MP \sqrt{\varepsilon}^{|l_2|} \varepsilon (\sqrt{\varepsilon})^{-|l_2|} \lambda \EP  \varepsilon \lambda  \]
and hence
\begin{equation}\label{termg}
|g\circ\sigma|_{\alpha,\tilde{L}} \MP \varepsilon \lambda. 
\end{equation}
In the same way,
\begin{equation}\label{termf}
|\tilde{f}\circ\sigma|_{\alpha,\tilde{L}} \MP \varepsilon\mu.
\end{equation}
We will further decompose $h$ and $[f]_\omega$ as follows. Without loss of generality, we can assume that $h(0)=0$. Then by Taylor's formula, we can expand $h$ at $I=0$ at order $2$:
\[ h(\sqrt{\varepsilon}I)=\sqrt{\varepsilon}\varpi\cdot I_1+\varepsilon AI_1\cdot I_1+\varepsilon BI_2\cdot I_2+\varepsilon CI_1\cdot I_2+\varepsilon\sqrt{\varepsilon}R_h(I)  \]
where 
\[ R_h(I)=2^{-1}\int_{0}^{1}(1-t)^2\nabla^3 h(t\sqrt{\varepsilon} I)(I,I,I) dt, \]
$A$ and $B$ are square matrix of size respectively $d$ and $m$, and $C$ is a matrix of size $m$ times $d$. The rectangular term represented by $CI_1\cdot I_2$ is slightly inconvenient, so we will assume:

\medskip

(A.1) $C=0$.

\medskip

Then we expand $[f]_\omega$ at $I=0$ at order $0$:
\[ [f]_\omega(\theta_2,\sqrt{\varepsilon}I)=[f]_\omega(\theta_2,0)+\sqrt{\varepsilon}\int_{0}^{1}\nabla_I [f]_\omega(\theta_2,t\sqrt{\varepsilon} I)I dt=\tilde{V}(\theta_2)+\sqrt{\varepsilon}R_{[f]_\omega}(\theta_2,I). \]   
and so
\[ \varepsilon^{-1}[f]_\omega\circ\sigma=\varepsilon^{-1}\tilde{V}+\sqrt{\varepsilon}^{-1}R_{[f]_\omega}=V+\sqrt{\varepsilon}^{-1}R_{[f]_\omega}. \]
Thus we have
\begin{eqnarray*}
\mathcal{H}(\theta,I) & = & \sqrt{\varepsilon}^{-1}\varpi\cdot I_1+AI_1\cdot I_1+BI_2\cdot I_2+V(\theta_2) \\
& + & (\sqrt{\varepsilon}R_h(I)+\sqrt{\varepsilon}^{-1}R_{[f]_\omega}(\theta_2,I)+\varepsilon^{-1} g(\theta_2,\sqrt{\varepsilon}I))+\varepsilon^{-1}f(\theta,\sqrt{\varepsilon}I). 
\end{eqnarray*}
Applying Lemma~\ref{deriv} (respectively with $s=3$ and $s=1$), we obtain $|R_h|_{\alpha,L'} \MP 1$ and $|R_{[f]_\omega}|_{\alpha,L'}\MP \varepsilon$, for $L'=\tilde{L}/2$, so that the first two terms in the parenthesis above (which are the Hamiltonians independent of $\theta_1$) are of order $\sqrt{\varepsilon}$. However, the last term in this parenthesis is of order $\lambda\geq \sqrt{\varepsilon}$ by~\eqref{termg}, so we define
\[ R= \lambda^{-1}\sqrt{\varepsilon}R_h +\lambda^{-1}\sqrt{\varepsilon}^{-1}R_{[f]_\omega} +\lambda^{-1}\varepsilon^{-1}g\circ\sigma, \] 
so that the Hamiltonian in the parenthesis above is $\lambda R$, with $|R|_{\alpha,L'}\MP 1$. Let us also write
\[ F=\varepsilon^{-1}\mu^{-1}\tilde{f}\circ\sigma \]
so that the last term in the expression of $\mathcal{H}$ is $\mu F$, with $|F|_{\alpha,L'}\leq |F|_{\alpha,\tilde{L}}\MP 1$ by~\eqref{termf}. Finally, we obviously have $|V|_{\alpha,L'} \leq 1$.

Now recalling the dependence of $\lambda$ and $\mu$ in the notation, we have obtained a Hamiltonian $\mathcal{H}=\mathcal{H}_{\lambda,\mu}$, defined on $\mathcal{D}_1$, of the form
\begin{equation}\label{Hres}
\begin{cases} 
\mathcal{H}_{\lambda,\mu}(\theta,I)=\mathcal{H}_{\lambda}(\theta_2,I)+\mu F(\theta,I), \quad |F|_{\alpha,L'}\MP 1  \\
\mathcal{H}_{\lambda}(\theta_2,I)=\mathcal{H}^{av}(\theta_2,I)+\lambda R(\theta_2,I), \quad |R|_{\alpha,L'} \MP 1\\
\mathcal{H}^{av}(\theta_2,I)=\sqrt{\varepsilon}^{-1}\varpi\cdot I_1+AI_1\cdot I_1+BI_2\cdot I_2+V(\theta_2), \quad |V|_{\alpha,L'} \leq 1.
\end{cases}
\end{equation}  
It will be convenient to consider $\mathcal{H}_{\lambda,\mu}$ as a Hamiltonian depending on two independent parameters $\lambda$ and $\mu$, and as a rule the notation $\lambda(\varepsilon)$ and $\mu(\varepsilon)$ will be used only when we want to recall that they both depend on $\varepsilon$ as in~\eqref{mu}. One has to consider $\mathcal{H}_{\lambda,\mu}$ as an arbitrary $\mu$-perturbation of $\mathcal{H}_{\lambda}=\mathcal{H}_{\lambda,0}$, and $\mathcal{H}_{\lambda}$ as a special $\lambda$-perturbation of the ``averaged" system $\mathcal{H}^{av}$ (special because the perturbation $R$ is independent of $\theta_1$). Note that the averaged system can be further decomposed as a sum of two Hamiltonians
\[ \mathcal{H}^{av}(\theta_2,I)=K(I_1)+P(\theta_2,I_2),   \]
where $K(I_1)=\sqrt{\varepsilon}^{-1}\varpi\cdot I_1+AI_1\cdot I_1$ is a completely integrable system on $\mathcal{D}_1^d=\T^d \times B_1^d$ and $P(\theta_2,I_2)=BI_2\cdot I_2+V(\theta_2) $ is a mechanical system (or a ``multidimensional pendulum") on $\mathcal{D}_1^m=\T^m \times B_1^m$. We make the following assumption on the mechanical system:

\medskip

(A.2) The matrix $B$ is positive definite (or negative definite), and the function $V : \T^m \rightarrow \R$ has a non-degenerated maximum (or minimum).

\medskip

Without loss of generality, we may assume that $V$ reaches its maximum at $\theta_2=0$, so that $O=(0,0)\in \mathcal{D}_1^m$ is a hyperbolic fixed point for the Hamiltonian flow generated by $P$. This in turns implies that, for any $I^* \in B_1^d$, the set $\mathcal{T}(I^*)=\{I_1=I_1^*\} \times O$ is a $d$-dimensional torus invariant for the averaged system, and it is hyperbolic in the sense that it has $C^1$ stable and unstable manifolds
\[ W^{\pm}(\mathcal{T}(I^*))=\{I_1=I_1^*\} \times W^{\pm}(O) \]
where $W^{\pm}(O)$ are the stable and unstable manifolds of $O$, which are Lagrangian. In particular, the torus $\mathcal{T}(0)$ is quasi-periodic with frequency $\sqrt{\varepsilon}^{-1}\varpi$.

Now this picture is easily seen to persist if we move from $\mathcal{H}^{av}$ to $\mathcal{H}_{\lambda}$. Indeed, since $\mathcal{H}_{\lambda}$ is still independent of $\theta_1$, the level sets of $I_1$ are still invariant, hence for a given $I_1^* \in B_1^d$, the Hamiltonian flow generated by the restriction of $\mathcal{H}_{\lambda}$ to $\{I_1=I_1^*\} \times \mathcal{D}_1^m$ (considered as a flow on $\mathcal{D}_1^m$ depending on $I_1^*$) is a $\lambda$-perturbation of the Hamiltonian flow generated by $P$: as a consequence it has a hyperbolic fixed point $O_\lambda(I_1^*) \in \mathcal{D}_1^m$ which is $\lambda$-close to $O$, for $\lambda$ small enough. Hence $\mathcal{T}_\lambda(I^*)=\{I_1=I_1^*\} \times O_\lambda(I_1^*)$ is invariant under the Hamiltonian flow of $\mathcal{H}_{\lambda}$, and it is hyperbolic with Lagrangian stable and unstable manifolds
\[ W^{\pm}(\mathcal{T}_\lambda(I^*))=\{I_1=I_1^*\} \times W^{\pm}(O_\lambda(I_1^*)). \] 
The torus $\mathcal{T}_\lambda(0)$ is still quasi-periodic with frequency $\sqrt{\varepsilon}^{-1}\varpi$. 

\begin{remark}\label{rem}
Let us point out here that in \cite{LMS03}, it is stated incorrectly that the size of the Hamiltonian we called $\lambda R$ is of order $\sqrt{\varepsilon}$ (as we mentioned above, this Hamiltonian is composed of three terms and only two of which are of order $\sqrt{\varepsilon}$, the last one being of order $\lambda(\varepsilon)$). Now in \cite{LMS03}, they made use of normal forms taken in \cite{Pos93} which contains estimates that are less accurate than ours (as we already explained, $\lambda(\varepsilon)$ is only of order one in \cite{Pos93}) and consequently these estimates do not allow to show the existence of the invariant torus for the system we called $\mathcal{H}_{\lambda}$.   
\end{remark}

Our next assumption concerns the persistence of the torus $\mathcal{T}_\lambda(0)$, as well as its stable and unstable manifolds, when we move from $\mathcal{H}_{\lambda}$ to $\mathcal{H}_{\lambda,\mu}$:  

\medskip

(A.3) For any $0 \leq \lambda \MP 1$ and $0 \leq \mu \MP \lambda$, the system $\mathcal{H}_{\lambda,\mu}$ has an invariant torus $\mathcal{T}_{\lambda,\mu}$, with $\mathcal{T}_{\lambda,0}=\mathcal{T}_{\lambda}=\mathcal{T}_\lambda(0)$, of frequency $\sqrt{\varepsilon}^{-1}\varpi$, with $C^1$ stable and unstable manifolds $W^{\pm}(\mathcal{T}_{\lambda,\mu})$ which are exact Lagrangian graphs over fixed relatively compact domains $U^{\pm} \subseteq \T^n$. Moreover, $W^{\pm}(\mathcal{T}_{\lambda,\mu})$ are $\mu$-close to $W^{\pm}(\mathcal{T}_{\lambda})$ for the $C^{1}$-topology.

\medskip

Let us denote by $S^{\pm}_{\lambda,\mu}$ generating functions for $W^{\pm}(\mathcal{T}_{\lambda,\mu})$ over $U^{\pm}$, that is if $V^{\pm}=U^{\pm} \times B_1$, then 
\[ W^{\pm}(\mathcal{T}_{\lambda,\mu}) \cap V^{\pm}=\{ (\theta_*,I_*) \in V^{\pm} \; | \: I_*=\partial_\theta S^{\pm}_{\lambda,\mu}(\theta_*)  \} \]
where $S^{\pm}_{\lambda,\mu} : U^{\pm} \rightarrow \R$ are $C^2$ functions. Since  $W^{\pm}(\mathcal{T}_{\lambda,\mu})$ are $\mu$-close to $W^{\pm}(\mathcal{T}_{\lambda})$ for the $C^{1}$-topology, the first derivatives of the functions $S^{\pm}_{\lambda,\mu}$ are $\mu$-close to the first derivatives of $S^{\pm}_{\lambda}=S^{\pm}_{\lambda,0}$ for the $C^1$-topology.

Then, in order to evaluate the splitting, we need the existence of orbits which are homoclinic to $\mathcal{T}_{\lambda,\mu}$:

\medskip

(A.4) For any $0 \leq \lambda \MP 1$ and $0 \leq \mu \MP \lambda$, the set $W^+(\mathcal{T}_{\lambda,\mu}) \cap W^-(\mathcal{T}_{\lambda,\mu}) \setminus \mathcal{T}_{\lambda,\mu}$ is non-empty. 

\paraga We can finally define the notions of splitting matrix and splitting angles, and state our results. 

The set $W^+(\mathcal{T}_{\lambda,\mu}) \cap W^-(\mathcal{T}_{\lambda,\mu}) \setminus \mathcal{T}_{\lambda,\mu}$ is invariant, so it consists of orbits of the Hamiltonian system defined by $\mathcal{H}_{\lambda,\mu}$. Let $\gamma_{\lambda,\mu}$ be one of this homoclinic orbit, and $p_{\lambda,\mu}=\gamma_{\lambda,\mu}(0)=(\theta_{\lambda,\mu},I_{\lambda,\mu})$ a homoclinic point. Since $p_{\lambda,\mu}$ is a homoclinic point, $\theta_{\lambda,\mu} \in U^+ \cap U^-$ and $\partial_\theta S^{+}_{\lambda,\mu}(\theta_{\lambda,\mu})=\partial_\theta S^{-}_{\lambda,\mu}(\theta_{\lambda,\mu})$. Then we can define the splitting matrix $M(\mathcal{T}_{\lambda,\mu},p_{\lambda,\mu})$ of $\mathcal{T}_{\lambda,\mu}$ at the point $p_{\lambda,\mu}$, as the symmetric square matrix of size $n$
\[ M(\mathcal{T}_{\lambda,\mu},p_{\lambda,\mu})= \partial_\theta^2 (S_{\lambda,\mu}^+-S_{\lambda,\mu}^-)(\theta_{\lambda,\mu}). \]
Moreover, for $1\leq i \leq n$, we define the splitting angles $a_i(\mathcal{T}_{\lambda,\mu},p_{\lambda,\mu})$ as the eigenvalues of the matrix $M(\mathcal{T}_{\lambda,\mu},p_{\lambda,\mu})$. Note that since the homoclinic point belongs to a homoclinic orbit, at least one of these angles is necessarily zero.

\begin{theorem}\label{thmsplit0}
Let $\mathcal{H}_{\lambda,\mu}$ be as in~\eqref{Hres}, and assume that (A.2), (A.3) and (A.4) are satisfied. Then, with the previous notations, we have the estimates
\[ |a_i(\mathcal{T}_{\lambda,\mu},p_{\lambda,\mu})| \MP \mu, \quad 1 \leq i \leq d.\]
\end{theorem} 

We have used the fact that the stable and unstable manifolds are exact Lagrangian graphs over some domains in $\T^n$ in order to define the splitting matrix and splitting angles, but in fact only the Lagrangian property (and not the exactness nor the graph property) is necessary to make those definitions (see \cite{LMS03}).

Now the solutions of the Hamiltonian system defined by $H \circ \Phi$ differs from those of $\mathcal{H}_{\lambda(\varepsilon),\mu(\varepsilon)}$ only by a time change, so $\mathcal{T}_{\lambda(\varepsilon),\mu(\varepsilon)}$ is still an invariant hyperbolic torus for $H \circ \Phi$, with the same stable and unstable manifolds. Coming back to our original system, the torus $T_{\varepsilon}=\Phi(\mathcal{T}_{\lambda(\varepsilon),\mu(\varepsilon)})$ is hyperbolic for $H$, with stable and unstable manifolds $W^{\pm}(T_{\varepsilon})=\Phi(W^{\pm}(\mathcal{T}_{\lambda(\varepsilon),\mu(\varepsilon)}))$, and for $\gamma_{\varepsilon}=\Phi(\gamma_{\lambda(\varepsilon),\mu(\varepsilon)})$ and $p_{\varepsilon}=\Phi(p_{\lambda(\varepsilon),\mu(\varepsilon)})$ we can define a splitting matrix $M(\mathcal{T}_{\varepsilon},p_{\varepsilon})$ and splitting angles $a_i(\mathcal{T}_{\varepsilon},p_{\varepsilon})$ for $1\leq i \leq n$.  

\begin{theorem}\label{thmsplit}
Let $H$ be as in~\eqref{Hnonlin}, with $r=2\sqrt{\varepsilon}$ satisfying~\eqref{thr2}. Assume that (A.1) is satisfied, and that (A.2), (A.3) and (A.4) are satisfied for the Hamiltonian $\mathcal{H}_{\lambda(\varepsilon),\mu(\varepsilon)}$. Then, with the previous notations, we have the estimates
\[ |a_i(\mathcal{T}_{\varepsilon},p_{\varepsilon})| \MP \sqrt{\varepsilon}\left(1+(\Delta_{\omega}^*(\cdot\sqrt{\varepsilon}^{-1}))^{-1}\right) \exp\left(-\cdot \Delta_{\omega}^*(\cdot\sqrt{\varepsilon}^{-1})^{\frac{1}{\alpha}}\right), \quad 1 \leq i \leq d.\]
\end{theorem}

\begin{corollary}\label{thmsplitD}
Let $H$ be as in~\eqref{Hnonlin}, with $\omega \in \Omega_d(\gamma,\tau)$ and $r=2\sqrt{\varepsilon}$ satisfying~\eqref{thr3}. Assume that (A.1) is satisfied, and that (A.2), (A.3) and (A.4) are satisfied for the Hamiltonian $\mathcal{H}_{\lambda(\varepsilon),\mu(\varepsilon)}$. Then, with the previous notations, we have the estimates
\[ |a_i(\mathcal{T}_{\varepsilon},p_{\varepsilon})| \MP \sqrt{\varepsilon}\left(1+(\cdot\gamma^{-2}\varepsilon)^{\frac{1}{2(1+\tau)}}\right) \exp\left(-\cdot (\cdot\gamma^{2}\varepsilon^{-1})^{\frac{1}{2\alpha(1+\tau)}}\right), \quad 1 \leq i \leq d.\]
\end{corollary}

Let us note that the exponent in the exponential factor in our corollary is not far from being optimal, at least for $\alpha>1$, $d=n-1$ and for a badly approximable vector $\omega \in \Omega_{n-1}(\gamma,n-2)$: indeed, in this case we have an exponentially small upper bound with the exponent $(2\alpha (n-1))^{-1}$, whereas in \cite{Mar05}, a sequence of $\varepsilon_j$-perturbations, with $\varepsilon_j$ going to zero as $j$ goes to infinity, is constructed such that the perturbed system has an invariant hyperbolic torus of dimension $d=n-1$, with $d-1=n-2$ splitting angles which have an exponential small lower bound with the exponent $(2(\alpha-1)(n-2))^{-1}$.
 
Let us now briefly discuss the validity of our assumptions (A.1), (A.2), (A.3) and (A.4), referring to \cite{LMS03} for more details. Concerning (A.1), in principle it is just a simplifying assumption and it can be removed, though we shall not try to do it here. The assumption (A.2) is crucial as it ensures that the averaged system has an invariant hyperbolic tori. Then, using classical KAM theory, (A.3) follows from (A.2) under usual assumptions (there are many references in the analytic case, but we do not know any for non-analytic but sufficiently regular systems). Now concerning  the existence of homoclinic orbits (A.4), this is a general assumption as follows: using a variational argument one can prove that the mechanical system has orbits homoclinic to the hyperbolic fixed point and so the averaged system has orbits homoclinic to the hyperbolic torus $\mathcal{T}(0)$, then assuming that the stable and unstable manifolds of $\mathcal{T}(0)$ intersect transversely along one of this homoclinic orbit inside the energy level, the assumption (A.4) is satisfied, that is this homoclinic orbit for the averaged system can be continued to a homoclinic orbit for the full system.

Note that the results in Theorem~\ref{thmsplit0} and Theorem~\ref{thmsplit} are just upper bounds on $d$ splitting angles. These splitting angles can be actually equal to zero, as we did not assume that the stable and unstable manifolds intersect transversely inside the energy level. As we already mentioned, this transversality assumption implies (A.4) provided (A.1), (A.2) and (A.3) are satisfied, and moreover, under this assumption, one can prove that they are exactly $d$ splitting angles which are non-zero and exponentially small as the other angles can only be polynomially small (we recall that at least one of them is zero, but we could have avoided this situation by taking a Poincaré section and studied the associated discrete system).

\paraga Let us now give some details concerning the proof of Theorem~\ref{thmsplit0} and Theorem~\ref{thmsplit}, which follows from simple lemmas (the proof of which, or references, can be found in \cite{LMS03}). 

Consider first the Hamiltonian $H_\lambda$, with its hyperbolic invariant torus $\mathcal{T}_\lambda=\mathcal{T}_{\lambda,0}=\mathcal{T}_\lambda(0)$ with stable and unstable manifolds
\[ W^{\pm}(\mathcal{T}_\lambda)=\{I_1=0\} \times W^{\pm}(O_\lambda(0)). \] 
Let us denote by $S^{\pm}_{\lambda}=S^{\pm}_{\lambda,0}$ the corresponding generating functions, and by $p_\lambda=p_{\lambda,0}$ the homoclinic point. From the expression of $W^{\pm}(\mathcal{T}_\lambda)$, one immediately has $\partial_{\theta_1}S^{\pm}_\lambda=0$. Let us write $M_\lambda=M(\mathcal{T}_\lambda,p_\lambda)$. The following lemma is then obvious.

\begin{lemma}\label{lemma1}
The matrix $M_\lambda$ admits the following block decomposition:
\[ M_\lambda=\left( \begin{array}{cc}
0 & 0  \\
0 & M_{\lambda}^{\perp}  \\
\end{array} \right)\]
where $M_{\lambda}^{\perp}=\partial^2_{\theta_2}(S^+_\lambda-S^-_\lambda)(\theta_\lambda)$ is a square symmetric matrix of size $m$. 
\end{lemma}

Note that the matrix $M_\lambda^\perp$ always have zero as an eigenvalue (because we are in a continuous setting), and that the fact that $W^{\pm}(\mathcal{T}_\lambda)$ intersect transversely along the homoclinic orbit inside the energy level is equivalent to zero being a simple eigenvalue of $M_\lambda^\perp$. 

Now let us come back to the Hamiltonian $H_{\lambda,\mu}$, and let us write $M_{\lambda,\mu}=M(\mathcal{T}_{\lambda,\mu},p_{\lambda,\mu})$. From assumption (A.3) the second derivatives of the functions $S^{\pm}_{\lambda,\mu}$ are $\mu$-close to the second derivatives of $S^{\pm}_{\lambda}$ for the $C^0$-topology. This immediately implies the following lemma, where we also denote by $|\,.\,|$ the norm induced by the supremum norm on the space of matrices.

\begin{lemma}\label{lemma2}
We have the estimate
\[ |M_{\lambda,\mu}-M_\lambda| \MP \mu. \]
\end{lemma} 

Then, in order to use the previous lemma, we need to know how the eigenvalues of a symmetric matrix vary under perturbation. This is the content of the next lemma, where we denote by $d$ the distance on $\R^n$ induced by the supremum norm.

\begin{lemma}\label{lemma3}
Let $A$ and $A'$ be two symmetric matrices, with spectrum $\mathrm{Spec}(A)$ and $\mathrm{Spec}(A')$. Then we have the estimate
\[ d(\mathrm{Spec}(A),\mathrm{Spec}(A'))\leq |A-A'|. \] 
\end{lemma}

Let us remark that the above lemma is quite specific to symmetric matrices, and the fact that the splitting matrix is indeed symmetric in the general case (in our restricted case, this is obvious by definition) ultimately comes from the Lagrangian character of the stable and unstable manifolds.

The proof of Theorem~\ref{thmsplit0} is now a straightforward consequence of Lemma~\ref{lemma1}, Lemma~\ref{lemma2} and Lemma~\ref{lemma2}.

Let us now come back to the Hamiltonian $H$ as in~\eqref{Hnonlin}. It is related to the Hamiltonian $\mathcal{H}_{\lambda(\varepsilon),\mu(\varepsilon)}$ by a scaling transformation and by a symplectic transformation $\Phi$. The effect of the scaling transformation is simply to multiply the splitting matrix by $\sqrt{\varepsilon}$. The effect of the symplectic transformation is more complicated to describe, but the overall effect is contained in the following lemma, where we denote $M_\varepsilon=M(\mathcal{T}_{\varepsilon},p_{\varepsilon})$.

\begin{lemma}\label{lemma4}
There exist two invertible square matrix $B$ and $C$ of size $n$ such that
\[M_\varepsilon=\sqrt{\varepsilon}B M_{\lambda(\varepsilon),\mu(\varepsilon)}C\]
where $B$ and $C$ satisfy
\[ |B|\MP 1+(\Delta_{\omega}^*(\cdot\sqrt{\varepsilon}^{-1}))^{-1} , \quad |C^{-1}|\MP 1+(\Delta_{\omega}^*(\cdot\sqrt{\varepsilon}^{-1}))^{-1}. \]
\end{lemma}

The matrices $B$ and $C$ depends on the differential of the transformation $\Phi$ (we refer to \cite{LMS03}, Proposition 1.5.4 for an explicit expression), and the estimates on $B$ and $C$ follows from the estimates on $\partial_\theta \Phi_\theta$ and $\partial_\theta \Phi_I$ contained in~\eqref{estnonlindist} in Theorem~\ref{thmnonlin}.
 
Finally, we need yet another lemma from linear algebra.

\begin{lemma}\label{lemma5}
Let $A$ and $A'$ two symmetric matrices, with $A'=BAC$ for two invertible square matrix $B$ and $C$ of size $n$. Assume that $A$ has $d$ eigenvalues in the interval $(-\delta,\delta)$. Then $A'$ has $d$ eigenvalues in the interval $(-\delta|B||C^{-1}|,\delta|B||C^{-1}|)$.
\end{lemma}

The proof of Theorem~\ref{thmsplit} is now an obvious consequence of Theorem~\ref{thmsplit0}, Lemma~\ref{lemma4} and Lemma~\ref{lemma5}.

\section{Proof of the main results}\label{s4}

This section is devoted to the proofs of Theorem~\ref{thmlin} and Theorem~\ref{thmnonlin}.

\paraga Let us first state an approximation lemma of an arbitrary vector by linearly independent periodic vectors, which was proved in \cite{BF12}. Recall that a vector $v \in \R^n\setminus\{0\}$ is called periodic if there exists a real number $t>0$ such that $tv \in \Z^n$. In this case, $T=\inf\{t>0 \; | \; tv \in \Z^n\}$ is called the period of $v$, and a periodic vector with period $T$ will be simply called $T$-periodic. It is easy to see that a vector is periodic if and only if its minimal rational subspace is one-dimensional.

Now consider an arbitrary vector $\omega \in \R^n\setminus\{0\}$, and let $d$ be the dimension of its minimal rational subspace $F=F_\omega$. For a given $Q\geq 1$, it is always possible to find a $T$-periodic vector $v \in F$, which is a $Q$-approximation in the sense that $|T\omega - Tv|\leq Q^{-1}$, and for which the period $T$ satisfies the upper bound $T \MP Q^{d-1}$: this is essentially the content of Dirichlet's theorem. Then it is not hard to see that there exist not only one, but $d$ linearly independent periodic vectors in $F$ which are $Q$-approximations. Moreover, one can obtain not only linearly independent vectors, but periodic vectors $v_1,\dots,v_d$ of periods $T_1, \dots,T_d$ such that the integer vectors $T_1v_1,\dots,T_dv_d$ form a $\Z$-basis of $\Z^n \cap F$. However, the upper bound on the associated periods $T_1, \dots, T_d$ is necessarily bigger than $Q^{d-1}$, and is given by a function that we call here $\Psi_\omega'$ (once again, see \cite{BF12} for more precise and general information, but note that there $\Psi_\omega'$ was denoted by $\Psi_\omega$ and $\Psi_\omega$, which we defined in~\eqref{psi}, was denoted by $\Psi_\omega'$). The main Diophantine result of \cite{BF12} is that this function $\Psi_\omega'$ is in fact equivalent to the function $\Psi_\omega$, up to constants and for $Q$ large enough. This gives the following result.

\begin{proposition}\label{dio}
Let $\omega \in \R^n\setminus\{0\}$. For any $Q\PS 1$, there exist $d$ periodic vectors $v_1, \dots, v_d$, of periods $T_1, \dots, T_d$, such that $T_1v_1, \dots, T_dv_d$ form a $\Z$-basis of $\Z^n \cap F$ and for $j\in\{1,\dots,d\}$,
\[ |\omega-v_j|\MP(T_j Q)^{-1}, \quad 1 \MP T_j \MP \Psi_\omega(Q).\]
\end{proposition}

For the proof, we refer to \cite{BF12}, Proposition $2.3$. The implicit constants depends only on $d$ and $\omega$ (the dependence on $\omega$ is through its norm $|\omega|$ and the discriminant of the lattice $\Z^n \cap F$).

Now a consequence of the fact that the vectors $T_1v_1, \dots, T_dv_d$ form a $\Z$-basis of $\Z^n \cap F$ is contained in the following corollary. For simplicity, we shall write $[\,\cdot\,]_{v_1,\dots,v_d}=[\cdots[\,\cdot\,]_{v_1}\cdots]_{v_d}$, where $[\,\cdot\,]_w$ has been defined for an arbitrary vector $w$ in~\eqref{ave}.

\begin{corollary}\label{cordio}
Under the assumptions of Proposition~\ref{dio}, let $l_\omega(I)=\omega\cdot I$ and $l_{v_j}(I)=v_j\cdot I$ for $j\in\{1,\dots,d\}$. For any $g\in C^{1}(\mathcal{D}_R)$, we have $[g]_\omega=[g]_{v_1,\dots,v_d}$ and therefore $\{g,l_\omega\}=0$ if and only if $\{g,l_{v_j}\}=0$ for any $j\in\{1,\dots,d\}$.
\end{corollary} 

\begin{proof}
Let $\Lambda=\Z^n \cap F$ and $d\vartheta$ be the Haar measure on the compact quotient group $F/\Lambda$. The flow $X_\omega^t$ leaves invariant the foliation on $\T^n$ induced by the trivial affine foliation on $\R^n$ defined by $F$. Each leaf of this foliation on $\T^n$ is diffeomorphic to $F/\Lambda$, and the restriction of $X_\omega^t$ to each leaf is uniquely ergodic. By Birkhoff's ergodic theorem, we have
\[ [g]_{\omega}=\lim_{s\rightarrow +\infty}\frac{1}{s}\int_{0}^{s}g\circ X_{\omega}^{t}dt=\int_{\vartheta \in F/\Lambda} g\circ X_{\vartheta}^{1} d\vartheta. \]
Also, as $T_jv_j$ is an integer vector, we have
\[ [g]_{v_j}=\lim_{s\rightarrow +\infty}\frac{1}{s}\int_{0}^{s}g\circ X_{v_j}^{t}dt=\frac{1}{T}\int_{0}^{T}g\circ X_{v_j}^{t}dt=\int_{0}^{1}g\circ X_{T_jv_j}^t dt. \]
Using the fact that the vectors $T_1v_1, \dots, T_dv_d$ form a $\Z$-basis of $\Lambda$, the first assertion follows easily from these expressions by a change of variables. Now if $\{g,l_{v_j}\}=0$ for any $j\in\{1,\dots,d\}$, then $g=[g]_{v_j}$ and so by the first assertion, $g=[g]_\omega$, which means that $\{g,l_\omega\}=0$. Conversely, if $\{g,l_\omega\}=0$, then $g=[g]_\omega$ and therefore $g=[g]_{v_1,\dots,v_d}$ by the first assertion. Since the maps $[\,\cdot\,]_{v_j}$ are projectors (that is $[\,\cdot\,]_{v_j,v_j}=[\,\cdot\,]_{v_j}$), this implies that 
\[ [g]_{v_d}=[g]_{v_1,\dots,v_d,v_d}=[g]_{v_1,\dots,v_d}=g, \]
and since they commute, we eventually find $[g]_{v_j}=g$ for any $j\in\{1,\dots,d\}$ and therefore $\{g,l_{v_j}\}=0$ for any $j\in\{1,\dots,d\}$.    
\end{proof}

\paraga Then we shall make use of the statement of Theorem~\ref{thmlin} (or Corollary~\ref{thmlinD}) in the particular case $d=1$, that is when $F$ is one-dimensional, which was proved in \cite{MS02}. As we already said, in this situation the vector is in fact periodic so we shall denote it by $v$, and for any non-zero integer vector $k \in F$, we have the lower bound $|k\cdot v| \geq T^{-1}$ and so we will have $\Delta_v^*(\varepsilon^{-1})\geq (T\varepsilon)^{-1}$ in the statement of Theorem~\ref{thmlin} (or $\tau=0$ and $\gamma=T^{-1}$ in the statement of Corollary~\ref{thmlinD}) for this particular case.

For subsequent use, we introduce another parameter $\nu>0$ and we consider the Hamiltonian
\begin{equation}\label{Hper}
\begin{cases} 
H(\theta,I)=l_v(I)+s(I)+u(\theta,I), \quad (\theta,I)\in \mathcal{D}_R, \\
Tv \in \Z^n, \quad |s|_{\alpha,L} \leq \nu, \quad |u|_{\alpha,L} \leq \varepsilon. 
\end{cases}
\end{equation}
Let us define $\delta=(2d)^{-1}R$. Then we have the following result.

\begin{proposition}\label{ana}
Let $H$ be as in \eqref{Hper}, and assume that
\begin{equation}\label{condana}
\varepsilon \MP \nu, \quad T\nu \MP 1.
\end{equation}
Then there exists a symplectic map $\Phi \in G^{\alpha,L'}(\mathcal{D}_{R-\delta},\mathcal{D}_{R})$, with $C \EP 1$ and $L'=CL$, such that
\[ H\circ\Phi=l_v+s+[u]_v+u'+\tilde{u}, \quad \{u',l_v\}=0 \]
with the estimates
\[ |\Phi-\mathrm{Id}|_{\alpha,L'} \MP T\nu, \quad |u'|_{\alpha,L'} \MP \varepsilon T\nu, \quad |\tilde{u}|_{\alpha,L'} \MP \varepsilon \exp(-\cdot (T\nu)^{-1/\alpha}).  \]
\end{proposition}

For the Hamiltonian~\eqref{Hper}, the term $l_v$ is considered as unperturbed, and $s+u$ is the perturbation. Since we have assumed that $\varepsilon\MP\mu$, the size of the perturbation is of order $\mu$, but as $s$ is integrable, the size of the non-resonant part of the perturbation is of order $\varepsilon$. Note that if we are only interested in the periodic case, then one may take $s=0$ in~\eqref{Hper}, $\varepsilon=\nu$ and write $u=f$ in the statement of Proposition~\ref{ana}, and this gives exactly the statement of Theorem~\ref{thmlin}.

For a proof of Proposition~\ref{ana}, we refer to \cite{MS02}, Proposition $3.2$ (note that in \cite{MS02}, the size of the perturbation is called $\varepsilon$ instead of $\mu$, and the size of the non-resonant part of the perturbation is called $\varepsilon'$ instead of $\varepsilon$). The implicit constants in the above statement depends only on $n,R,\alpha$ and $L$.

Let us now state a simple algebraic property (in our restricted setting), which complements Proposition~\ref{ana}.

\begin{proposition}\label{anaad}
Under the assumptions of Proposition~\ref{ana}, suppose that $l_w$ is a linear integrable Hamiltonian such that $\{u,l_w\}=0$. Then, in the conclusions of Proposition~\ref{ana}, we have $\{u',l_w\}=0.$
\end{proposition}  

This property has been used several times (this was first used in \cite{Bam99} and \cite{Pos99b}) and is valid under much more general assumptions, we refer to \cite{BN09} for a proof in the analytic case (see also \cite{Bou11}). Of course, this is a purely algebraic property, and is valid regardless of the regularity of the system.

\paraga Now we can finally prove Theorem~\ref{thmlin}, which is a straightforward consequence of the more flexible proposition below.

\begin{proposition}\label{proplin}
Let $H$ be as in \eqref{Hlin}, and $Q\geq 1$. If 
\begin{equation}\label{condlin}
Q \PS 1, \quad \varepsilon \MP \Delta_\omega(Q)^{-1},
\end{equation}
then there exists a symplectic map $\Phi \in G^{\alpha,\tilde{L}}(\mathcal{D}_{R/2},\mathcal{D}_R)$, with $\tilde{L}\EP L$, such that 
\[ H \circ \Phi = l_\omega +[f]_\omega+ g + \tilde{f}, \quad \{g,l_\omega\}=0  \]
with the estimates
\[ |\Phi-\mathrm{Id}|_{\alpha,\tilde{L}} \MP Q^{-1}, \quad |g|_{\alpha,\tilde{L}}\MP \varepsilon Q^{-1}, \quad |\tilde{f}|_{\alpha,\tilde{L}}\MP \varepsilon \exp\left(-\cdot Q^{-1/\alpha}\right). \]
\end{proposition}

\begin{proof}[Proof of Theorem~\ref{thmlin}]
We choose 
\[ Q = \Delta_\omega^*(\cdot\varepsilon^{-1}) \]
with a well-chosen implicit constant so that the second part of~\eqref{condlin} is satisfied. Proposition~\ref{proplin} with this value of $Q$ implies Theorem~\ref{thmlin}, as the first part of~\eqref{condlin} is satisfied by the threshold~\eqref{thr1}.  
\end{proof}

It remains to prove Proposition~\ref{proplin}.

\begin{proof}[Proof of Proposition~\ref{proplin}]
Recall that we are considering $H$ as in \eqref{Hlin}. Since $Q \PS 1$ by the first part of~\eqref{condlin}, we can apply Proposition~\ref{dio}: there exist $d$ periodic vectors $v_1, \dots, v_d$, of periods $T_1, \dots, T_d$, such that $T_1v_1, \dots, T_dv_d$ form a $\Z$-basis of $\Z^n \cap F$ and for $j\in\{1,\dots,d\}$,
\[ |\omega-v_j|\MP(T_j Q)^{-1}, \quad 1 \MP T_j \MP \Psi_\omega(Q) .\] 
For $j\in\{1,\dots,d\}$, let us define 
\[ s_j=l_\omega-l_{v_j}, \quad \nu_j \EP (T_j Q)^{-1} \]
with a suitable implicit constant so that $|s_j|_{\alpha,L} \leq \nu_j$. Note that $l_\omega=l_{v_j}+s_j$, and that $T_j\nu_j \EP Q^{-1}$. Let us further define $L_j=C^jL$ and $R_j=R-j\delta$ so that in particular $R_d=R/2$.

Then we claim that for all $1 \leq j \leq d$, there exists a symplectic map $\Phi_j \in G^{\alpha,L_j}(\mathcal{D}_{R_j},\mathcal{D}_R)$ such that 
\[ H\circ\Phi_j=l_\omega+[f]_{v_1,\dots,v_j}+g_j+f_j, \quad \{g_j,l_{v_i}\}=0 \]
for any $1 \leq i \leq j$, with the estimates
\[ |\Phi_j-\mathrm{Id}|_{\alpha,L_j} \MP Q^{-1}, \quad |g_j|_{\alpha,L_j} \MP \varepsilon Q^{-1}, \quad |f_j|_{\alpha,L_j} \MP \varepsilon \exp(-\cdot Q^{-1/\alpha}).  \]
The proof of the proposition follows from this claim: it is sufficient to let $\Phi=\Phi_d$, $g=g_d$, $\tilde{f}=f_d$ and $\tilde{L}=L_d$, since from Corollary~\ref{cordio}, $[f]_{v_1,\dots,v_d}=[f]_\omega$ and $\{g,l_{v_i}\}=0$ for $1 \leq i \leq d$ is equivalent to $\{g,l_{\omega}\}=0$.

Now let us to prove the claim by induction. For $j=1$, writing $l_\omega=l_{v_1}+s_1$, this is nothing but Proposition~\ref{ana} (up to a change of notations, namely $s_1$ instead of $s$, $f$ instead of $u$, $g_1$ instead of $u'$ and $f_1$ instead of $\tilde{u}$). So assume the statement holds true for some $1 \leq j \leq d-1$, and let us prove it is true for $j+1$. By the inductive assumption, there exists a symplectic map $\Phi_j \in G^{\alpha,L_j}(\mathcal{D}_{R_j},\mathcal{D}_R)$ such that 
\[ H\circ\Phi_j=l_\omega+[f]_{v_1,\dots,v_j}+g_j+f_j, \quad \{g_j,l_{v_i}\}=0,\] 
for any $1 \leq i \leq j$, with the estimates
\[ |\Phi_j-\mathrm{Id}|_{\alpha,L_j} \MP Q^{-1}, \quad |g_j|_{\alpha,L_j} \MP \varepsilon Q^{-1}, \quad |f_j|_{\alpha,L_j} \MP \varepsilon \exp(-\cdot Q^{-1/\alpha}).  \]
Now let us consider the Hamiltonian
\[ H\circ\Phi^j-f_j=l_\omega+[f]_{v_1,\dots,v_j}+g_j=l_{v_{j+1}}+s_{j+1}+[f]_{v_1,\dots,v_j}+g_j=l_{v_{j+1}}+s_{j+1}+u_j \]
with $u_j=[f]_{v_1,\dots,v_j}+g_j$, so $|u_j|_{\alpha,L_j} \MP (\varepsilon+\varepsilon Q^{-1}) \MP \varepsilon$. We want to apply Proposition~\ref{ana} to this Hamiltonian, and we observe that \eqref{condlin} implies \eqref{condana}: indeed, $Q \PS 1$ implies $T_j\nu_j \EP Q^{-1} \MP 1$, whereas 
\[ \varepsilon \MP \Delta_\omega(Q)^{-1} \EP (Q\Psi_\omega(Q))^{-1}\MP (T_jQ)^{-1} \EP \nu_j. \]
So we can apply Proposition~\ref{ana}: there exists a symplectic map $\Phi^{j+1} \in G^{\alpha,L_{j+1}}(\mathcal{D}_{R_{j+1}},\mathcal{D}_{R_j})$ such that 
\begin{eqnarray*}
( H\circ\Phi_j-f_j)\circ\Phi^{j+1} & = & l_{v_{j+1}}+s_{j+1}+[u_j]_{v_{j+1}}+u_j'+\tilde{u}_j \\
& = & l_\omega + [[f]_{v_1,\dots,v_j}+g_j]_{v_{j+1}}+u_j'+\tilde{u}_j \\
& = & l_\omega+[f]_{v_1,\dots,v_{j+1}}+[g_j]_{v_{j+1}}+u'_{j}+\tilde{u}_j
\end{eqnarray*}
with $\{u'_{j},l_{v_{j+1}}\}=0$ and the estimates
\[ |\Phi^{j+1}-\mathrm{Id}|_{\alpha,L_{j+1}} \MP T_{j+1}\nu_{j+1} \EP Q^{-1} \]
and
\[ |u'_{j}|_{\alpha,L_{j+1}} \MP \varepsilon Q^{-1}, \quad |\tilde{u}_j|_{\alpha,L_{j+1}} \MP \varepsilon \exp(-\cdot (T_{j+1}\nu_{j+1})^{-1/\alpha}) \EP \varepsilon \exp(-\cdot Q^{-1/\alpha}).  \]
Obviously, we have $|[g_j]_{v_{j+1}}|_{\alpha,L_j} \MP \varepsilon Q^{-1}$, so we set
\[ \Phi_{j+1}=\Phi_{j} \circ \Phi^{j+1}, \quad g_{j+1}=[g_{j}]_{v_{j+1}}+u'_j, \quad f_{j+1}=\tilde{u}_j+f_j\circ\Phi^{j+1}. \]
We have $\Phi_{j+1} \in G^{\alpha,L_{j+1}}(\mathcal{D}_{R_{j+1}},\mathcal{D}_{R_{j}})$, and since
\begin{equation}\label{eqide}
|\Phi^{j+1}-\mathrm{Id}|_{\alpha,L_{j+1}} \MP  Q^{-1} \MP 1,
\end{equation}
using Corollary~\ref{corcomp} we have the estimate
\[ |\Phi_{j+1}-\mathrm{Id}|_{\alpha,L_{j+1}} \leq |\Phi_j-\mathrm{Id}|_{\alpha,L_j}+|\Phi^{j+1}-\mathrm{Id}|_{\alpha,L_{j+1}} \MP Q^{-1}. \]
Now $\{g_j,l_{v_i}\}=0$ for $1 \leq i \leq j$ from the induction hypothesis, and obviously we have $\{[f]_{v_1,\dots,v_j},l_{v_i}\}=0$ for $1 \leq i \leq j$, hence $\{u_j,l_{v_i}\}=0$ for $1 \leq i \leq j$. Then if we apply Proposition~\ref{anaad} with $l_w=l_{v_i}$ for $1 \leq i \leq j$, we obtain that $\{u_j',l_{v_i}\}=0$ for $1 \leq i \leq j$. But $\{u_j',l_{v_{j+1}}\}=0$ and hence $\{u_j',l_{v_i}\}=0$ for $1 \leq i \leq j+1$. Moreover, we claim that $\{g_j,l_{v_i}\}=0$ for $1 \leq i \leq j$ implies $\{[g_{j}]_{v_{j+1}},l_{v_i}\}=0$ for $1 \leq i \leq j$, and as $\{[g_{j}]_{v_{j+1}},l_{v_{j+1}}\}=0$, this gives $\{[g_{j}]_{v_{j+1}},l_{v_i}\}=0$ for $1 \leq i \leq j+1$, and this eventually gives $\{g_{j+1},l_{v_i}\}=0$ for $1 \leq i \leq j+1$. To prove the claim, note that since $\{l_{v_i},l_{v_{j+1}}\}=0$ for $1\leq i \leq j$, $l_{v_i}=l_{v_i}\circ X_{v_{j+1}}^{t}$ and so
\begin{eqnarray*}
\{[g_j]_{v_{j+1}},l_{v_i}\} & = & \lim_{s\rightarrow +\infty}\frac{1}{s}\int_{0}^{s}\{g_j\circ X_{{v_{j+1}}}^{t},l_{v_i}\}dt \\
& = & \lim_{s\rightarrow +\infty}\frac{1}{s}\int_{0}^{s}\{g_j\circ X_{v_{j+1}}^{t},l_{v_i}\circ X_{v_{j+1}}^{t}\}dt \\
& = & \lim_{s\rightarrow +\infty}\frac{1}{s}\int_{0}^{s}\{g_j,l_{v_i}\} \circ X_{v_{j+1}}^{t}dt \\
& = & 0.
\end{eqnarray*} 
The estimate for $g_{j+1}$ is obvious:
\[ |g_{j+1}|_{\alpha,L_{j+1}} \leq |[g_j]_{v_{j+1}}|_{\alpha,L_j} + |u'_{j}|_{\alpha,L_{j+1}} \MP \varepsilon Q^{-1}.   \]
Finally, from Lemma~\ref{propcomp} and~\eqref{eqide}, we have
\[ |f_j\circ\Phi_{j+1}|_{\alpha,L_{j+1}} \leq |f_j|_{\alpha,L_j} \]
and therefore
\[ |f_{j+1}|_{\alpha,L_{j+1}} \leq |\tilde{u}_j|_{\alpha,L_{j+1}} + |f_j|_{\alpha,L_j} \MP \varepsilon \exp(-\cdot Q^{-1/\alpha}). \]  
This proves the claim, and ends the proof of the proposition.
\end{proof}

\paraga Let us now prove Theorem~\ref{thmnonlin}, which is an easy consequence of Theorem~\ref{thmlin}. As before, this will be deduced from the more flexible proposition below.

\begin{proposition}\label{propnonlin}
Let $H$ be as in \eqref{Hnonlin}, and $Q\geq 1$. If 
\begin{equation}\label{condnonlin}
Q \PS 1, \quad \varepsilon \leq r^2, \quad r \MP \Delta_\omega(Q)^{-1}, \quad r\leq R,
\end{equation}
then there exists a symplectic map $\Phi \in G^{\alpha}(\mathcal{D}_{r/2},\mathcal{D}_r)$, such that
\[ H \circ \Phi = h +[f]_\omega+ g + \tilde{f}, \quad \{g,l_\omega\}=0.  \]
Moreover, for any $l=(l_1,l_2)\in \N^{2n}$, we have the estimates 
\begin{equation*}
|\Phi_I-\mathrm{Id}_I|_{\alpha,\tilde{L},l} \MP  rr^{-|l_2|}Q^{-1}, \quad |\Phi_\theta-\mathrm{Id}_\theta|_{\alpha,\tilde{L},l} \MP r^{-|l_2|}Q^{-1}
\end{equation*}
and
\begin{equation*}
|g|_{\alpha,\tilde{L},l} \MP  r^2r^{-|l_2|}Q^{-1}, \quad |\tilde{f}|_{\alpha,\tilde{L},l}\MP r^2r^{-|l_2|}\exp\left(-\cdot Q^{-1/\alpha}\right).
\end{equation*}
\end{proposition}
 
\begin{proof}[Proof of Theorem~\ref{thmnonlin}]
We choose 
\[ Q = \Delta_\omega^*(\cdot r^{-1}) \]
with a well-chosen implicit constant so that the third part of~\eqref{condnonlin} is satisfied. Proposition~\ref{propnonlin} with this value of $Q$ implies Theorem~\ref{thmnonlin}, as the other conditions of~\eqref{condnonlin} are satisfied by~\eqref{thr2}. 
\end{proof}

Now let us prove Proposition~\ref{propnonlin}.

\begin{proof}[Proof of Proposition~\ref{propnonlin}]
To analyze our Hamiltonian $H$ in the domain $\mathcal{D}_r$, which is a neighbourhood of size $r$ around the origin, we rescale the action variables using the map
\[ \sigma : (\theta,I) \longmapsto (\theta,rI) \]
which sends the domain $\mathcal{D}_{1}$ onto $\mathcal{D}_r$, the latter being included in $\mathcal{D}_R$ by the last part of~\eqref{condnonlin}. Let 
\[ H'=r^{-1}(H\circ\sigma)\] 
be the rescaled Hamiltonian, so $H'$ is defined on $\mathcal{D}_{1}$ and reads
\[ H'(\theta,I)=r^{-1}H(\theta,r I)=r^{-1}h(r I)+r^{-1}f(\theta,r I), \quad (\theta,I)\in\mathcal{D}_{1}. \]
Without loss of generality, we assume that $h(0)=0$. Now using Taylor's formula we can expand $h$ around the origin to obtain
\begin{equation*}
h(r I)  = r\omega \cdot I+r^2\int_{0}^{1}(1-t)\nabla^2 h(tr I)(I,I) dt=r\omega\cdot I+r^2 h'(I) 
\end{equation*}
and so we can write
\[ H'=l_\omega+f' \]
with 
\[ f'=r h'+r^{-1}(f\circ\sigma). \]
Now we know that $|f|_{\alpha,L}\leq\varepsilon$, so that the $|r^{-1}(f\circ\sigma)|_{\alpha,L}\leq r^{-1}\varepsilon$, and using the second part of~\eqref{condnonlin}, $|r^{-1}(f\circ\sigma)|_{\alpha,L}\leq r$. Moreover, applying Lemma~\ref{deriv} (with $s=2$) we have $|h'|_{\alpha,L/2}\MP |h|_{\alpha,L} \MP 1$, so that $|rh'|_{\alpha,L/2} \MP r$. This eventually gives $|f'|_{\alpha,L/2} \MP r$. Now we will apply Proposition~\ref{proplin} to the Hamiltonian $H'=l_\omega+f'$, defined on the domain $\mathcal{D}_1$ and such that $|f'|_{\alpha,L/2} \MP r$. The first and third part of condition~\eqref{condnonlin} imply condition~\eqref{condlin} with $1$ instead of $R$, $L/2$ instead of $L$ and $r$ instead of $\varepsilon$, so that there exists a symplectic map $\Phi' \in G^{\alpha,\tilde{L}}(\mathcal{D}_{1/2},\mathcal{D}_1)$, with $\tilde{L}\EP L$, such that
\[ H' \circ \Phi' = l_\omega +[f']_\omega+ g' + \tilde{f}', \quad \{g',l_\omega\}=0  \]
with the estimates 
\[ |\Phi'-\mathrm{Id}|_{\alpha,\tilde{L}} \MP Q^{-1}, \quad |g'|_{\alpha,\tilde{L}}\MP r Q^{-1}, \quad |\tilde{f}'|_{\alpha,\tilde{L}}\MP r \exp\left(-\cdot Q^{-1/\alpha}\right). \]
Note that $[f']_\omega=[rh'+r^{-1}(f\circ\sigma)]_\omega=rh'+r^{-1}[f]_\omega\circ\sigma$ so the transformed Hamiltonian can be written as
\begin{equation*}
H' \circ \Phi'=l_\omega+r h'+r^{-1}[f]_\omega\circ\sigma+g'+\tilde{f}'.
\end{equation*}
Now scaling back to our original coordinates, we define $\Phi=\sigma \circ \Phi' \circ \sigma^{-1}$, therefore $\Phi : \mathcal{D}_{r/2} \longrightarrow \mathcal{D}_{r}$ and 
\begin{eqnarray*}
H\circ\Phi & = & r H'\circ\Phi' \circ \sigma^{-1} \\
& = & r (l_\omega+r h'+r^{-1}[f]_\omega\circ\sigma+g'+\tilde{f}') \circ \sigma^{-1} \\
& = & (r l_\omega+r^2 h') \circ \sigma^{-1} +[f]_\omega+ r g'\circ \sigma^{-1} +r\tilde{f}'\circ \sigma^{-1}.  
\end{eqnarray*}
Observe that $(r l_\omega+r^2 h') \circ \sigma^{-1}=h$, so we may set
\[ g=rg'\circ \sigma^{-1}, \quad \tilde{f}=r\tilde{f}'\circ \sigma^{-1}, \]
and write
\[ H\circ\Phi=h+[f]_{\omega}+g+\tilde{f}. \]
The equality $\{g,l_\omega\}=0$ is obvious. Now let $l=(l_1,l_2)\in \N^{2n}$. Observe that from the definition of $\sigma$ and $g$,
\[ \partial^l g=\partial_\theta^{l_1}\partial_I^{l_2}g=rr^{-|l_2|}(\partial^lg')\circ \sigma^{-1}\] 
so
\[ |\partial^l g|_{C^0(\mathcal{D}_{r/2})} \leq rr^{-|l_2|} |\partial^l g'|_{C^0(\mathcal{D}_{\rho/2})} \]
and therefore
\[ |g|_{\alpha,\tilde{L},l} \leq rr^{-|l_2|}|g'|_{\alpha,\tilde{L},l} \MP r^2r^{-|l_2|} Q^{-1}.\]
Replacing $g'$ by $\tilde{f}'$ in the above argument, we obtain 
\[ |\tilde{f}|_{\alpha,\tilde{L},l} \leq rr^{-|l_2|}|\tilde{f}'|_{\alpha,\tilde{L},l} \MP r^2r^{-|l_2|}\exp\left(-\cdot Q^{-1/\alpha}\right).\]
Finally, writing $\Phi=(\Phi_\theta,\Phi_I)$ and $\Phi'=(\Phi'_\theta,\Phi'_I)$, we observe that $\Phi_\theta=\Phi_\theta' \circ \sigma^{-1}$ and $\Phi_I=r\Phi_I' \circ \sigma^{-1}$ which immediately gives the estimates for $\Phi_\theta$ and $\Phi_I$. This concludes the proof.
\end{proof}

\appendix

\section{Technical estimates}\label{app1}

In this section we recall some technical estimates concerning Gevrey functions, taken from \cite{MS02}, that are used in the proofs. 

Note that our Gevrey norm differs from the one in \cite{MS02}, where they used
\[ ||H||_{G^{\alpha,L}(\mathcal{D}_R)}=\sum_{l\in \N^{2n}}L^{|l|\alpha}(l!)^{-\alpha}|\partial^l H|_{C^{0}(\mathcal{D}_R)} < \infty  \]
instead of~\eqref{norm}. However these norms are ``almost equivalent" in the sense that obviously we have $|H|_{G^{\alpha,L}(\mathcal{D}_R)} \leq ||H||_{G^{\alpha,L}(\mathcal{D}_R)}$ while, for instance, $||H||_{G^{\alpha,L/2}(\mathcal{D}_R)} \leq 2^\alpha(2^\alpha-1)^{-1}|H|_{G^{\alpha,L}(\mathcal{D}_R)}$. So this change only affects the implicit constants.

First recall that our proof of Theorem~\ref{thmlin} proceed by induction on $1\leq j \leq d$, starting with the case $d=1$ which is exactly Proposition~\ref{ana}, and the transformation in Theorem~\ref{thmlin} is a composition of transformations given by Proposition~\ref{ana}. Therefore we shall have to estimate the Gevrey norm of a composition of functions, and unlike classical $C^0$ norms used in the analytic case, these estimates are not completely trivial.

\begin{lemma}\label{propcomp}
There exists a constant $C$ depending on $n, \alpha$ and $L$ such that for any $0<\delta<R$, if $F\in G^{\alpha,L}(\mathcal{D}_R,\R^{2n})$ and $\Phi\in G^{\alpha,L'}(\mathcal{D}_{R'},\mathcal{D}_R)$, with $L'=CL$ and $R'=R-\delta$, and if $|\Phi-\mathrm{Id}|_{\alpha,L'}\MP 1$, then $F\circ\Phi \in G^{\alpha,L'}(\mathcal{D}_{R'},\R^{2n})$ and we have the estimate
\[ |F\circ\Phi|_{\alpha,L'} \leq |F|_{\alpha,L}. \]
\end{lemma}

This lemma is contained in the statement of Corollary A.1, Appendix A.2, in \cite{MS02}, in the case where $F$ is a real-valued function, but it extends immediately to vector-valued functions, by applying it components by components. The implicit constant also depends on $\delta$, but the statement will be used for a value of $\delta$ depending only on $d$ and $R$. 

Here's a direct consequence that we will also use.

\begin{corollary}\label{corcomp}
There exists a constant $C$ depending on $n, \alpha$ and $L$ such that for any $0<\delta<R$, if $\Psi\in G^{\alpha,L}(\mathcal{D}_R,\R^{2n})$ and $\Phi\in G^{\alpha,L'}(\mathcal{D}_{R'},\mathcal{D}_R)$, with $L'=CL$ and $R'=R-\delta$, and if $|\Phi-\mathrm{Id}|_{\alpha,L'}\MP 1$, then $\Psi\circ\Phi \in G^{\alpha,L'}(\mathcal{D}_{R'},\R^{2n})$ and we have the estimate
\[ |\Psi\circ\Phi-\mathrm{Id}|_{\alpha,L'} \leq |\Psi-\mathrm{Id}|_{\alpha,L}+|\Phi-\mathrm{Id}|_{\alpha,L'}. \]
\end{corollary}

Let us also state a straightforward lemma which says that the Gevrey norm of a function controls the Gevrey norm of its derivatives, provided we restrict the parameter $L$ (in the statement below, we simply choose $L/2$). 

\begin{lemma}\label{deriv} 
Let $f\in G^{\alpha,L}(\mathcal{D}_R)$. For any $s \in \N$, there exists a constant $c$ depending on $s,\alpha$ and $L$ such that
\[ \sup_{l\in\N^{2n},\,|l|=s}|\partial^l f|_{\alpha,L/2}\leq c|f|_{\alpha,L}. \]
\end{lemma}

\addcontentsline{toc}{section}{References}
\bibliographystyle{amsalpha}
\bibliography{GNFG}

\end{document}